\documentclass[11pt,a4paper]{amsart}

\usepackage[T1]{fontenc}
\usepackage[utf8]{inputenc}
\usepackage{lmodern}

\usepackage[english]{babel}
\usepackage{amsmath,amsfonts,amssymb,amsthm}

\usepackage[pdftex]{graphicx}

\usepackage[bottom=3cm, top=3cm, left=3.5cm, right=3.5cm]{geometry}

\usepackage[foot]{amsaddr}
\usepackage[alphabetic, abbrev]{amsrefs}

\usepackage{mathtools}

\usepackage{bm}	
\DeclareMathAlphabet{\mathbbold}{U}{bbold}{m}{n}	

\theoremstyle{plain}
\newtheorem{theorem}{Theorem}
\newtheorem*{theorem*}{Theorem}
\newtheorem{proposition}[theorem]{Proposition}

\newtheorem{lemma}[theorem]{Lemma}
\newtheorem*{lemma*}{Lemma}

\theoremstyle{definition}
\newtheorem{definition}[theorem]{Definition}

\theoremstyle{remark}
\newtheorem{remark}[theorem]{Remark}
\newtheorem{example}[theorem]{Example}

\usepackage[bookmarksdepth=3]{hyperref}
\usepackage{xcolor}
\hypersetup{
    colorlinks,
    linkcolor={red!50!black},
    citecolor={blue!50!black},
    urlcolor={blue!80!black}
}

\usepackage{enumitem}
\setlist[itemize]{leftmargin=1cm}

\newcommand{\R}{\mathbb{R}}

\newcommand{\Heis}{\mathbb{H}}
\newcommand{\distr}{\mathcal{D}}

\newcommand{\eps}{\varepsilon}

\newcommand{\sdistr}{\mathcal{F}}					

\newcommand{\dd}{\,{\mathrm d}}
\newcommand{\db}{{\mathrm d}}
\newcommand{\car}{C}

\newcommand{\pt}{\partial}

\newcommand{\vol}{\Omega}
\newcommand{\quasi}{\zeta}

\newcommand{\ads}{H}
\newcommand{\Sinads}{\widetilde{H}}

\let\oldtocsection=\tocsection 
\let\oldtocsubsection=\tocsubsection

\renewcommand{\tocsection}[2]{\hspace{0em}\oldtocsection{#1}{#2}}
\renewcommand{\tocsubsection}[2]{\hspace{1em}\oldtocsubsection{#1}{#2}}

\makeatletter
\@namedef{subjclassname@2020}{%
  \textup{2020} Mathematics Subject Classification}
\makeatother
\author[Davide Barilari]{Davide Barilari$^\flat$}
\address{$^\flat\,$Dipartimento di Matematica ``Tullio Levi-Civita'',
  Universit\`a degli Studi di Padova, Via Trieste 63,
  Padova, Italy.}
\email{\href{mailto:davide.barilari@unipd.it}{davide.barilari@unipd.it}}

\author[Karen Habermann]{Karen Habermann$^\sharp$}
\address{$^\sharp\,$Department of Statistics, University of Warwick, Coventry, CV4 7AL, United Kingdom.}
\email{\href{mailto:karen.habermann@warwick.ac.uk}{karen.habermann@warwick.ac.uk}}

\title[Intrinsic sub-Laplacian for hypersurface]{Intrinsic
  sub-Laplacian for hypersurface in a contact sub-Riemannian manifold}


\date{\today}
\begin{document}

\begin{abstract}
  We construct and study the intrinsic sub-Laplacian,
  defined outside the set of characteristic points,
  for a smooth hypersurface embedded in a contact sub-Riemannian manifold.
  We prove that, away from characteristic points, the intrinsic
  sub-Laplacian arises as the limit of Laplace--Beltrami operators built
  by means of Riemannian approximations to the sub-Riemannian
  structure using the Reeb vector field. We carefully analyse three
  families of model cases for this setting obtained by considering
  canonical hypersurfaces embedded in model spaces for contact
  sub-Riemannian manifolds. In these model
  cases, we show that
  the intrinsic sub-Laplacian is stochastically complete and in
  particular, that the stochastic process induced by the intrinsic
  sub-Laplacian almost surely does not hit characteristic points.
\end{abstract}

\maketitle

\setcounter{tocdepth}{1}
\tableofcontents
\section{Introduction}
Recent years have seen increased activity in the study of hypersurfaces
embedded in contact sub-Riemannian manifolds, with
notable subtleties as well as distinctions
compared to the
Riemannian setting arising in the presence of characteristic points.
These are points on the hypersurface where the tangent space coincides with
the contact hyperplane.

First works in this direction have concerned the study of geometry of
hypersurfaces in Heisenberg groups, and more generally in Carnot
groups, in particular related to 
a notion of horizontal mean curvature and isoperimetric
inequalities. For a review of
these topics, see~\cites{DGN07,CDPT} and references
therein.

For surfaces in the Heisenberg group
the horizontal mean curvature, which blows up at
characteristic points, is locally integrable
with respect to the sub-Riemannian
perimeter measure, as shown by Danielli, Garofalo and Nhieu.  Their
conjecture given in~\cite{DGN} that around isolated characteristic
points
the horizontal mean curvature
is also locally integrable with respect to the Riemannian induced
measure is verified by Rossi~\cite{tommaso} for characteristic points which
are isolated and mildly degenerate.
Rizzi and Rossi~\cite{RR} give an example
for a domain in the Heisenberg group where a higher-order coefficient
in the asymptotic expansion for the heat content of
smooth non-characteristic domains blows up at an isolated
characteristic point.

Diniz and Veloso~\cite{DV}, assuming absence of characteristic
points,
and Balogh, Tyson and Vecchi~\cites{BTC,BTCcorr}, allowing for
characteristic points, introduce a
Gauss--Bonnet theorem for surfaces in the Heisenberg group, which
is extended by
Veloso~\cite{veloso} to surfaces without characteristic points in general
three-dimensional contact sub-Riemannian manifolds.
A Gauss--Bonnet
theorem recovering topological
information concentrated around the characteristic points
is obtained by Grong,
Hidalgo Calderón and Vega-Molino~\cite{GHV} for surfaces in
three-dimensional contact sub-Riemannian
manifolds.

The work~\cite{BBC} analyses the metric structure, particularly near
characteristic points, induced on surfaces
embedded
in three-dimensional contact sub-Riemannian
manifolds, and \cite{BBCH} introduces and studies properties of a
canonical stochastic process on surfaces
in three-dimensional contact sub-Riemannian
manifolds which exhibits different
behaviours near an elliptic characteristic point and a hyperbolic
characteristic point.

The present article aims to initiate further studies of hypersurfaces
embedded in higher-dimensional contact sub-Riemannian manifolds. We
intrinsically construct a sub-Laplacian on hypersurfaces
in contact sub-Riemannian manifolds, which for surfaces in
three-dimensional contact sub-Riemannian manifolds gives rise to the 
generator of the stochastic process obtained in~\cite{BBCH} by
means of
Riemannian approximations, and we use our analysis to propose model
cases for this setting.
Some notions such as horizontal connectivity, horizontal connection and
horizontal mean curvature on hypersurfaces in sub-Riemannian manifolds
are studied by Tan and Yang~\cite{TY}.

We close this literature overview by highlighting that whilst
the powerful convex surface theory in three-dimensional contact
topology, that is, without additionally equipping the contact
structure with a fibre inner product, 
has been introduced by Giroux~\cite{giroux1}, the theory of convex
hypersurfaces in higher-dimensional contact topology is still a relatively
new endeavour, see Honda and Huang~\cite{CHT}.

\medskip

Let $M$ be a smooth manifold of dimension $2n+1$ for $n\geq 1$, let
$\distr$ be a contact structure on $M$, and let $g$ be a smooth fibre
inner product on $\distr$. Since this gives rise to a contact manifold
$(M,\distr)$
and as $(\distr,g)$ defines a sub-Riemannian structure on
the manifold $M$, the
triple $(M,\distr,g)$ is called a contact sub-Riemannian
manifold. Throughout, we shall assume that there exists a
global one-form $\omega$ on $M$ such that $\distr=\ker\omega$ and
$\omega\wedge\left(\db\omega\right)^{n}\neq 0$. Such a global one-form
$\omega$ is called contact form for the contact structure
$\distr$. The existence of a contact form $\omega$ ensures that the
manifold $M$ is orientable as it can then be oriented by the
volume form
\begin{equation}\label{eq:defnOmega}
  \vol=\omega\wedge\left(\db\omega\right)^{n}\;.
\end{equation}
We shall further assume that the one-form $\omega$ is normalised such that
\begin{equation}\label{eq:normalise}
  \left.(\db\omega)^{n}\right|_{\distr}=n!\,\mathrm{vol}_{g}\;,
\end{equation}
with $\mathrm{vol}_{g}$ denoting the volume form on the
distribution $\distr$ induced by the fibre inner product $g$.
The
Reeb vector field $X_{0}$ on $M$
associated with the contact form $\omega$ is uniquely characterised
by requiring $\omega(X_{0})=1$ and $\db\omega(X_{0},\cdot)=0$.
Subject to the normalisation condition~(\ref{eq:normalise}), a
fixed sub-Riemannian manifold $(M,\distr,g)$ admits a unique Reeb
vector field $X_0$ as there exists a unique one-form $\omega$ on $M$
which both defines the contact structure $\distr$ and
satisfies~(\ref{eq:normalise}).

Let $S$ be an orientable hypersurface embedded in the contact
manifold $(M,\distr)$. We denote by $\car(S)$ the set of
characteristic points of $S$,
namely the set of points $x\in S$ such that
$T_{x}S=\distr_{x}$.
Observe that $\car(S)$ is a closed subset of $S$, which implies that
$S\setminus\car(S)$ is a well-defined hypersurface in $M$.
Outside the set of characteristic points, that is, on
$S\setminus\car(S)$, we define the distribution $\sdistr=\distr\cap
TS$ which by construction has corank one in the tangent bundle
of the hypersurface $S\setminus\car(S)$.
Let $\quasi$ be the one-form on $S\setminus\car(S)$ obtained by
restricting the one-form $\omega$ defined on $M$ to
$S\setminus\car(S)$, and note that the
distribution $\sdistr$ is given as
$\sdistr=\ker\quasi$.

A crucial observation to be made at this stage is that the case $n=1$
needs to be treated differently from the case $n>1$. Indeed, for
$n=1$, the manifold $M$ has dimension three, the hypersurface $S$
is a two-dimensional surface and the distribution $\sdistr$ is a line
field, which is always integrable. On the other hand, if $n>1$, then
$\sdistr$ is a rank $2n-1$ distribution on
a $2n$-dimensional hypersurface. As discussed in more details in
Section~\ref{sec:hyp}, when $n>1$, the distribution $\sdistr$ is
always bracket generating as a result of $\sdistr$ defining a
quasi-contact structure on $S\setminus\car(S)$. In particular, the
kernel $\left.\ker\db\quasi\right|_\sdistr$ has dimension one.

In this article, we construct an intrinsic sub-Laplacian on
$S\setminus\car(S)$ obtained by taking the divergence of the
horizontal gradient on $S\setminus\car(S)$ with respect to a volume
form for $S\setminus\car(S)$ which is the
restriction of the volume form $\vol$ on $M$.
We show that the
intrinsic sub-Laplacian arises as the limit of Laplace--Beltrami operators
built by
means of Riemannian approximations to the sub-Riemannian structure
using the Reeb vector field.
We further determine the radial part of the constructed
intrinsic sub-Laplacian explicitly 
for canonical hypersurfaces in the sphere
$S^{2n+1}$ and the anti-de Sitter space $\ads^{2n+1}$ both equipped with
standard sub-Riemannian contact structures as well as in the
higher-dimensional Heisenberg group $\Heis^n$, which together
constitute model spaces for our setting.

The construction of the intrinsic sub-Laplacian
presented below carries over for any alternative normalisation
condition fixing
the one-form $\omega$ in place of~(\ref{eq:normalise}). The associated
sub-Laplacian then still arises as the limit of Laplace--Beltrami
operators, except that the Reeb vector field used to define the
Riemannian approximations is uniquely characterised in terms of the
new normalisation. As can be observed later, any choice of
normalisation which only changes the one-form $\omega$ by a constant
gives rise to the same sub-Laplacian as obtained subject to the
condition~(\ref{eq:normalise}).

\subsection{Intrinsic sub-Laplacian on hypersurface}
We start with the construction of a sub-Laplacian on the
embedded hypersurface
$S\setminus\car(S)$ which is intrinsic to the contact sub-Riemannian manifold
$(M,\distr,g)$ and the normalisation condition~(\ref{eq:normalise}).
We then state the result that it emerges as the limit of
Laplace--Beltrami operators. This particularly implies that the operator
$\Delta_0$ constructed in~\cite{BBCH} on surfaces in three-dimensional
contact sub-Riemannian manifolds coincides with the intrinsic
sub-Laplacian defined in this article for the case $n=1$.

The sub-Riemannian normal in
$(M,\distr,g)$ to the hypersurface $S$ away from the set
$\car(S)$ of characteristic points is formed from directions
contained in the contact structure $\distr$ and orthogonal to the
distribution $\sdistr$. Once the orientations of $S$ and $M$ are
fixed, we have a unique unit and normal vector field $N$ compatible
with the orientations, which is defined as follows.
\begin{definition}\label{defn:sRnormal}
  The sub-Riemannian normal vector field $N$ along
  the hypersurface $S\setminus\car(S)$ in
  the contact sub-Riemannian manifold $(M,\distr,g)$ is the
  unit-length vector field in the distribution $\distr$, that is,
  \begin{equation}\label{defn:sRnormal1}
    \omega(N)=0\quad\text{and}\quad
    g(N,N)=1\;,
  \end{equation}
  such that, for any vector field $Y$ on $S\setminus\car(S)$ and in the
  distribution $\sdistr$,
  \begin{equation}\label{defn:sRnormal2}
    g(N,Y)=0\;,
  \end{equation}
  and, for any positively oriented local orthonormal frame
  $(Z_{1},\dots,Z_{2n})$ for $S\setminus\car(S)$, the frame
  $(N,Z_{1},\dots,Z_{2n})$ for $M$ is a positively oriented.
\end{definition}

Using the volume form $\vol$ on $M$ given
by~(\ref{eq:defnOmega}) and the sub-Riemannian normal vector field
$N$ along $S\setminus\car(S)$, we define a volume form $\mu$ on
$S\setminus\car(S)$ with respect to which we later take the divergence
when constructing the intrinsic sub-Laplacian on $S\setminus\car(S)$.
\begin{definition}\label{defn:hypervol}
  Let $\mu$ be the volume form defined on $S\setminus\car(S)$ as
  \begin{displaymath}
    \mu=\iota_{N}\vol\;,
  \end{displaymath}
  that is, the contraction of the form $\vol$ with the vector field
  $N$ restricted to $S\setminus\car(S)$.
\end{definition}
From the compatibility of $N$ with the orientations of $M$ and $S$, it
follows that $\mu=\iota_{N}\vol$ is positive on $S\setminus\car(S)$,
meaning it has positive values when evaluated on positively
oriented orthonormal frames.

\smallskip

The final ingredient needed before we can introduce the intrinsic
sub-Laplacian of a
smooth function $f\colon S\setminus\car(S)\to\R$ is the horizontal
gradient $\nabla_S f$.
\begin{definition}
  Let $f\colon S\setminus\car(S)\to\R$ be a smooth function. The
  horizontal gradient $\nabla_S f$ of the function $f$ is the unique
  vector field in the distribution $\sdistr$, that is,
  \begin{displaymath}
    \quasi(\nabla_S f)=0\;,
  \end{displaymath}
  such that, for any vector field $Y$ in $\sdistr$,
  \begin{displaymath}
    g(\nabla_S f,Y)=\db f(Y)\;.
  \end{displaymath}
\end{definition}

In particular, with a local orthonormal frame
$(Y_1,\dots,Y_{2n-1})$ for $\sdistr$, we can write
\begin{equation}\label{eq:horgrad}
  \nabla_S f=\sum_{i=1}^{2n-1} (Y_i f)Y_i\;,
\end{equation}
which follows by noting that, for all $j\in\{1,\dots,2n-1\}$,
\begin{displaymath}
  g(\nabla_S f,Y_j)=
  \sum_{i=1}^{2n-1} (Y_i f)g(Y_i,Y_j)=Y_jf=\db f(Y_j)\;.
\end{displaymath}

The intrinsic sub-Laplacian $\Delta$ on $S\setminus\car(S)$
is constructed as the divergence with respect to the
volume form $\mu$ of the horizontal gradient $\nabla_S$.
\begin{definition}\label{def:intlap}
  The intrinsic sub-Laplacian $\Delta$ for a hypersurface $S\setminus\car(S)$
  embedded in a contact sub-Riemannian manifold $(M,\distr,g)$ is
  given by, for a smooth function $f\colon S\setminus\car(S)\to \R$,
  \begin{displaymath}
    \Delta f =\operatorname{div}_\mu\left(\nabla_S f\right)\;.
  \end{displaymath}
\end{definition}

The sub-Laplacian $\Delta$ defined on $S\setminus\car(S)$ arises as the limit
of Laplace--Beltrami operators on $S$ in the
following way. Let $\theta_0\colon T M\to \R$
  be the unique linear map such that, for the Reeb vector field $X_0$
  on $M$ and for any vector field $X$ in $\distr$,
  \begin{displaymath}
    \theta_0(X_0)=1
    \quad\text{and}\quad
    \theta_0(X)=0\;.
  \end{displaymath}
For $\eps>0$, we consider the Riemannian metric
$\overline{g}^\eps$ on $M$ obtained as
\begin{equation}\label{eq:Rapprox}
  \overline{g}^\eps=g\oplus
  \frac{1}{\eps^2}\left(\theta_0\otimes \theta_0\right)\;.
\end{equation}
We use $i$ for the inclusion map $i\colon S\to M$ and observe
that $i^\ast\overline{g}^\eps$ is the Riemannian metric on $S$ induced
by the Riemannian metric $\overline{g}^\eps$ on $M$.
The Laplace--Beltrami operator $\Delta^\eps$ of the
$2n$-dimensional Riemannian
manifold $(S,i^\ast\overline{g}^\eps)$ then
converges to the intrinsic sub-Laplacian $\Delta$ uniformly on compacts
as $\eps\to 0$.
\begin{theorem}\label{thm:limitlap}
  For any
  smooth function $f\in C_c^\infty(S\setminus\car(S))$
  compactly supported in $S\setminus\car(S)$, the functions
  $\Delta^\eps f$ converge uniformly on $S\setminus\car(S)$ to
  $\Delta f$ as $\eps\to 0$.
\end{theorem}

Since the operator $\Delta_0$ introduced in~\cite{BBCH} on surfaces in
three-dimensional contact sub-Riemannian manifolds is constructed as
the limit of Laplace--Beltrami operators
and subject to the normalisation condition
$\left.\db\omega\right|_\distr=-\mathrm{vol}_g$, it follows
from Theorem~\ref{thm:limitlap} that, for $n=1$,
the intrinsic sub-Laplacian $\Delta$ from Definition~\ref{def:intlap}
coincides with the operator $\Delta_0$. Whilst the additional sign in the
normalisation condition compared to~(\ref{eq:normalise}) flips the
direction of the Reeb vector field $X_0$, it does not affect the
operator $\Delta$ because the divergence remains unchanged for measures
differing by a non-zero constant factor.

\subsection{Hypersurfaces in contact
  sub-Riemannian model spaces}
In~\cite{BBCH}, the operator $\Delta_0$ is explicitly determined for
natural choices of surfaces in the three classes of model spaces for
three-dimensional sub-Riemannian structures. We extend these
considerations to higher dimensions by studying the
intrinsic sub-Laplacian $\Delta$ from Definition~\ref{def:intlap}
for canonical hypersurfaces in the three classes of model spaces for contact
sub-Riemannian manifolds. Moreover, we analyse the radial part of
the stochastic process with generator $\frac{1}{2}\Delta$ which is
sufficient to deduce that in all these cases the sub-Laplacian $\Delta$
defined away from
characteristic points is stochastically complete.
At the same time, the geometry induced on each hypersurface minus
characteristic points is not geodesically complete.

The model spaces for contact sub-Riemannian manifolds arise by
equipping the Euclidean space $\R^{2n+1}$, the sphere
$S^{2n+1}$ and the hyperboloid
$\ads^{2n+1}$, respectively, with a
standard contact structure $\distr$ and the following fibre
inner product $g$ on $\distr$.
For $\R^{2n+1}$, we choose $g$ such that $(\R^{2n+1},\distr,g)$ gives
rise to the higher-dimensional Heisenberg group $\Heis^n$.
For
the sphere $S^{2n+1}$ embedded in $\R^{2n+2}$, we choose $k\in\R$
positive and set, with $\langle\cdot,\cdot\rangle$ denoting the
Euclidean inner product on $\R^{2n+2}$,
\begin{displaymath}
  g(\cdot,\cdot)=\frac{1}{k^2}
  \left.\langle\cdot,\cdot\rangle\right|_\distr\;.
\end{displaymath}
This gives rise to a one-parameter family of model spaces with
underlying manifold $S^{2n+1}$ and parameter $k>0$. Similarly,
for the hyperboloid $\ads^{2n+1}$ embedded in $\R^{2n,2}$, we use
the flat Lorentzian metric $\eta$ on $\R^{2n,2}$ and $k\in\R$ positive to
define
\begin{displaymath}
  g(\cdot,\cdot)=\frac{1}{k^2}
  \left.\eta(\cdot,\cdot)\right|_\distr\;,
\end{displaymath}
which yields a one-parameter family of
model spaces with underlying manifold $\ads^{2n+1}$ and parameter $k>0$.
The model spaces for contact sub-Riemannian
manifolds are described in more details in Section~\ref{sec:model}.

The hypersurface which we consider embedded in $\R^{2n+1}$, in
$S^{2n+1}$ and in $\ads^{2n+1}$, respectively, serves as a model
hypersurface in the corresponding model space and can be identified with
$\R^{2n}$, $S^{2n}$ and $\Sinads^{2n}$, respectively, with a unique
characteristic point in the first and third case, and with two
antipodal characteristic points in the second case.
We refer to Section~\ref{sec:heis},
Section~\ref{sec:sphere} and Section~\ref{sec:hyper} for more details
on the choice of the model hypersurface. In our analysis of the model
cases, we first obtain
expressions for the volume form $\vol$ on $M$ and for
the sub-Riemannian normal
vector field $N$
to derive an expression for the volume form $\mu$ on the
hypersurface away from characteristic points.

\begin{proposition}\label{propn:vol}
  Let $(M,\distr,g)$ be a $(2n+1)$-dimensional
  contact sub-Riemannian model space. Set $I=(0,\frac{\pi}{k})$ if
  $M=S^{2n+1}$ associated with parameter $k>0$ and set $I=(0,\infty)$
  otherwise. Define $h_k\colon I\to \R$ by, for $r\in I$,
  \begin{displaymath}
    h_k(r)=
    \begin{cases}
      r    & \text{if }M=\R^{2n+1}\\
      k^{-1}\sin(kr) & \text{if }M=S^{2n+1}\\
      k^{-1}\sinh(kr)& \text{if }M=\ads^{2n+1}
    \end{cases}\;.
  \end{displaymath}
  For the model hypersurface $S$ in
  the model space $(M,\distr,g)$
  and in suitable coordinates $(r,\varphi_1,\dots,\varphi_{2n-1})$ for
  $S\setminus\car(S)$ with $r\in I$,
  $\varphi_1,\dots,\varphi_{2n-2}\in[0,\pi]$ and
  $\varphi_{2n-1}\in[0,2\pi)$, the volume form $\mu$ defined on
  $S\setminus\car(S)$ is given by
  \begin{displaymath}
    \mu=\frac{n!}{2}\left(h_k(r)\right)^{2n}
    \left(\prod_{i=1}^{2n-2}\left(\sin(\varphi_i)\right)^{2n-i-1}\right)
    \db r\wedge \db\varphi_1\wedge\dots\wedge\db\varphi_{2n-1}\;.
  \end{displaymath}
\end{proposition}

We observe that, except for a leading constant, the volume
forms induced on the Euclidean space $\R^{2n}$, the sphere $S^{2n}$ and
the hyperboloid $\Sinads^{2n}$
differ from the standard Riemannian volume forms by a factor of
$h_k$. This additional factor is the main reason why the radial
part of the stochastic process with generator $\frac{1}{2}\Delta$ is
of one order higher than in the model spaces for Riemannian manifolds
of the same topological dimension.
For a discussion on the radial process of Brownian motion on the model
Riemannian manifolds, see
e.g.\ Grigor'yan~\cite[Section~3.10]{grigoryan} and
Hsu~\cite[Section~3.3]{hsu}. The radial part of sub-Riemannian
Brownian motion in the setting of totally geodesic foliations is
studied in~\cite{BGKNT}.
\begin{theorem}\label{thm:radial}
  Let $(M,\distr,g)$ be a $(2n+1)$-dimensional
  contact sub-Riemannian model space. For the model hypersurface $S$ in
  the model space $(M,\distr,g)$, the radial part of the
  stochastic process with generator $\frac{1}{2}\Delta$
  on $S\setminus\car(S)$ is 
  \begin{itemize}
  \item the Bessel process of order $2n+1$ if $M=\R^{2n+1}$,
  \item a Legendre process of order $2n+1$ if $M=S^{2n+1}$,
  \item a hyperbolic Bessel process of order $2n+1$ if $M=\ads^{2n+1}$.
  \end{itemize}
\end{theorem}
Since a Bessel process of order $2n+1$ and a hyperbolic Bessel process
of order $2n+1$ for $n\geq 1$ almost surely neither hits the origin
nor explodes in finite time, and as a Legendre process of order $2n+1$
for $n\geq 1$ almost surely hits neither endpoint of the interval
$(0,\frac{\pi}{k})$, it is an immediate consequence of
Theorem~\ref{thm:radial} that in all model cases considered the
intrinsic sub-Laplacian $\Delta$ defined on $S\setminus\car(S)$ is
stochastically complete. On the
other hand, the geometry
induced on the hypersurface $S\setminus\car(S)$ is not geodesically
complete. This can
be seen by noting that a radial
ray, that is, a path along the radial direction emanating from
one of the
characteristic points, parameterised by arc
length is a geodesic which cannot be extended indefinitely towards the
characteristic point.

\subsection*{Organisation of the article.} In Section~\ref{sec:hyp},
we first provide an overview of contact sub-Riemannian manifolds and of
quasi-contact sub-Riemannian manifolds before showing that, for $n\geq
2$, a contact structure on a manifold $M$ of dimension $2n+1$ induces
a quasi-contact structure on a
hypersurface embedded in $M$ away from the set of characteristic
points. We illustrate this phenomenon by considering a canonical
hypersurface in the Heisenberg group $\Heis^2$.
In Section~\ref{sec:lap}, we describe the Laplace--Beltrami operators
$\Delta^\eps$ obtained by means of Riemannian approximations in a convenient
way which allows us to subsequently prove
Theorem~\ref{thm:limitlap}. We proceed by explicitly determining the
intrinsic sub-Laplacian for the considered hypersurface in
$\Heis^2$. In Section~\ref{sec:model}, we analyse model cases for
our setting, which results in proofs of Proposition~\ref{propn:vol}
and Theorem~\ref{thm:radial}.

\subsection*{Acknowledgements}Davide Barilari acknowledges support by the STARS Consolidator Grants 2021 ``NewSRG'' of the University of Padova.

\section{Hypersurfaces in contact sub-Riemannian manifolds}

\label{sec:hyp}
We start by providing a concise overview of contact sub-Riemannian manifolds
and of
quasi-contact sub-Riemannian manifolds. For more exhaustive
discussions, see e.g.~\cite{ABB19}, Boscain, Neel and
Rizzi~\cite[Section~10]{BNR},
and Charlot~\cite{charlot2002}. For an in-depth account on
contact geometry, one may consult Blair~\cite{blair} and
Geiges~\cite{geiges}. We then link contact sub-Riemannian
manifolds and quasi-contact
sub-Riemannian manifolds by showing that for a hypersurface $S$
in a manifold $M$ of dimension bigger than three, a contact structure on
$M$ induces a
quasi-contact structure on the hypersurface $S$ away from the set
$\car(S)$ of characteristic points.

A \emph{contact sub-Riemannian manifold} is a triple $(M,\distr, g)$
consisting of a smooth manifold $M$ with $\dim M=2n+1$ for $n\geq 1$,
a contact structure $\distr$ on $M$
and a
smooth fibre inner product $g$ defined on $\distr$.
The distribution $\distr$ is called a contact structure on $M$ if it
is locally defined as the kernel $\distr=\ker \omega$ of a one-form
$\omega$ on $M$ which satisfies the non-degeneracy condition
$\omega\wedge (\db\omega)^{n}\neq 0$. The latter is equivalent to
requiring that $\db\omega|_{\distr}$ is non-degenerate and implies
that the contact structure $\distr$ is a corank one distribution in the
tangent bundle $TM$. Recall we assume throughout that there
exists a global one-form $\omega$ defining the contact structure
$\distr$, which also induces an orientation on $M$ through the volume form
$\omega\wedge (\db\omega)^{n}$.

Moreover, we observe that for a smooth and positive function $f\colon M\to
(0,\infty)$, we have
\begin{displaymath}
  (f\omega)\wedge(\db(f\omega))^n=f^{n+1}\,\omega\wedge(\db\omega)^n
\end{displaymath}
as well as $\ker f\omega=\ker\omega$. Thus, the one-forms $\omega$ and
$f\omega$ define the same contact structure $\distr$ on $M$ and the
associated sub-Riemannian structures are equivalent. Due to
$\distr=\ker \omega$, we further obtain that
\begin{displaymath}
  \db(f\omega)|_{\distr}=f\dd\omega|_{\distr}\;.
\end{displaymath}
Hence, we can and do assume that the contact form $\omega$
satisfies the 
normalisation condition~(\ref{eq:normalise}), that is,
\begin{displaymath}
  \left.(\db\omega)^{n}\right|_{\distr}=n!\,\mathrm{vol}_{g}\;.
\end{displaymath}
The Reeb vector field $X_0$ on $M$ with respect to the one-form
$\omega$ normalised according to~(\ref{eq:normalise}) is uniquely
characterised by requiring that $\omega(X_{0})=1$ and
$\db\omega(X_{0},\cdot)=0$.

\smallskip

A \emph{quasi-contact sub-Riemannian manifold} is a triple
$(S,\sdistr, g)$ which consists of a smooth even-dimensional manifold $S$ 
where $\dim S=2n$ for $n\geq 2$,
a quasi-contact structure $\sdistr$ on $S$
and a smooth fibre inner product $g$
defined on $\sdistr$. A distribution
$\sdistr$ is called a quasi-contact structure on $S$ if it has corank
one in the tangent bundle $TS$ and is locally given
as $\sdistr=\ker\quasi$ for a one-form $\quasi$ on $S$ satisfying the 
non-degeneracy condition that $\db\quasi|_{\sdistr}$ has
one-dimensional kernel.
Note that since the manifold $S$ is of even dimension, the
distribution $\sdistr$
has odd rank and $\db\quasi|_{\sdistr}$
necessarily possesses a non-trivial kernel. Therefore, the above
non-degeneracy condition can be
understood as a minimal degeneracy assumption. If there exists a
global one-form $\quasi$ defining the quasi-contact structure
$\sdistr$, we call this one-form $\quasi$ a quasi-contact form.

The following property for quasi-contact structures is well-known but
we include its proof for completeness as it implies that the triple
$(S,\sdistr,g)$ introduced above is indeed a sub-Riemannian manifold.

\begin{lemma}\label{lem:bracketgen}
  A quasi-contact structure $\sdistr$ on a manifold $S$ is a bracket
  generating distribution on $S$.
\end{lemma}
Observe that we define
quasi-contact structures only in dimension $2n$ for $n\geq 2$,
which is an important condition here because a rank one distribution, that is,
a line field, is always integrable.
\begin{proof}[Proof of Lemma~\ref{lem:bracketgen}]
  Let $\quasi$ be a one-form locally defining the distribution
  $\sdistr$ through $\sdistr=\ker\quasi$. Since the kernel
  $\left.\ker\db\quasi\right|_\sdistr$
  has dimension one and as $\sdistr$ is of rank at least three, we can
  locally choose
  two vector fields $Y_1$ and $Y_2$ in $\sdistr$ such that
  $\db\quasi(Y_1,Y_2)$ is non-zero. Applying the Leibniz rule and the Cartan
  identity, we further obtain
  \begin{equation}\label{eq:provesH}
    \mathcal{L}_{Y_1}\left(\quasi(Y_2)\right)
    =\left(\mathcal{L}_{Y_1}\quasi\right)(Y_2)+\quasi\left([Y_1,Y_2]\right)
    =\db\quasi(Y_1,Y_2)
    +\iota_{Y_2}\dd\left(\quasi(Y_1)\right)+\quasi\left([Y_1,Y_2]\right)\;.
  \end{equation}
  For the vector fields $Y_1$ and $Y_2$, we have
  $\quasi(Y_1)=\quasi(Y_2)=0$ with $\db\quasi(Y_1,Y_2)$ being non-zero.
  The above
  identity then implies that $\quasi\left([Y_1,Y_2]\right)$ is non-zero
  because~(\ref{eq:provesH}) simplifies to
  \begin{displaymath}
    0=\db\quasi(Y_1,Y_2)+\quasi\left([Y_1,Y_2]\right)\;.
  \end{displaymath}
  It follows that the Lie bracket $[Y_1,Y_2]$ is not a vector field in
  $\sdistr$.
  As a quasi-contact structure is a
  distribution of rank 
  $2n-1$ on a manifold of dimension $2n$ for some $n\geq 2$,
  this concludes the proof.
\end{proof}

We now take a $(2n+1)$-dimensional contact sub-Riemannian manifold
$(M,\distr,g)$ for $n\geq 1$ with contact form $\omega$ satisfying the
normalisation condition~(\ref{eq:normalise}), and we consider an
orientable hypersurface $S$ embedded in $M$, where $\dim S=2n$. Recall
that the set $\car(S)$ of characteristic points of $S$ is given by
\begin{displaymath}
  \car(S)=\left\{x\in S:T_{x}S=\distr_{x}\right\}\;.
\end{displaymath}
The distribution $\sdistr$ defined on the hypersurface
$S\setminus\car(S)$ as
$\sdistr=\distr\cap TS$ has,
by construction, corank one in the tangent bundle of
$S\setminus\car(S)$.
As previously remarked, in the case $n=1$ studied
in~\cites{BBCH,BBC} the distribution $\sdistr$ is a line field, which
is always integrable. In contrast to this, the following result states
that, for $n >1$, the rank $2n-1$ distribution $\sdistr$ is a quasi-contact
structure on the $2n$-dimensional hypersurface $S\setminus\car(S)$.
Together with Lemma~\ref{lem:bracketgen}, this shows that the
distribution $\sdistr$ is bracket generating for $n>1$.
\begin{lemma}\label{lem:QConS}
  For $n\geq 2$, the distribution $\sdistr$ defined on
  $S\setminus\car(S)$ as $\sdistr=\distr\cap TS$
  is a quasi-contact structure
  on $S\setminus\car(S)$.
\end{lemma}
\begin{proof}
  For any $x\in S\setminus\car(S)$, we need to show that
  $\left.\db\quasi_x\right|_{\sdistr_x}\colon
    \sdistr_x\times\sdistr_x\to\R$, that is,
  \begin{displaymath}
    \left.\db\omega_x\right|_{\sdistr_x}\colon
    \sdistr_x\times\sdistr_x\to\R
  \end{displaymath}
  has one-dimensional kernel, which is a consequence of the following
  linear algebra observation.
  
  We recall the skew-symmetric bilinear form
  $\db\omega_x\colon\distr_x\times\distr_x\to\R$ is
  non-degenerate by
  assumption. This means that if $w\in\distr_x$ satisfies
  \begin{displaymath}
    \db\omega_x(v,w)=0
    \text{ for all }v\in\distr_x
  \end{displaymath}
  then $w=0$. Moreover, since $x\in S\setminus\car(S)$, we know that
  $\sdistr_x\subset\distr_x$ is a subspace of codimension one.
  The non-degeneracy of $\db\omega_x$ then implies that the
  orthogonal complement $\sdistr_x^\perp$ of $\sdistr_x$ defined with respect
  to the bilinear form $\db\omega_x$, that is,
  \begin{displaymath}
    \sdistr_x^\perp=\left\{w\in\distr_x:
    \db\omega_x(v,w)=0\text{ for all }v\in\sdistr_x\right\}
  \end{displaymath}
  has
  \begin{equation}\label{eq:dimofker}
    \dim\sdistr_x^{\perp}=\dim\distr_x-\dim\sdistr_x=1\;.
  \end{equation}
  Let $\xi\in\sdistr_x^{\perp}$ be non-zero.
  If $\xi\not\in\sdistr_x$, we would have
  $\distr_x=\sdistr_x\oplus\sdistr_x^\perp$ and $\xi$ would lie in the
  kernel of $\db\omega_x$. As this contradicts the non-degeneracy of
  $\db\omega_x$, it follows that $\xi\in\sdistr_x$. Therefore, we
  obtain
  \begin{displaymath}
    \sdistr_x^\perp=\ker\left.\db\quasi_x\right|_{\sdistr_x}\;,
  \end{displaymath}
  and the desired result follows from~(\ref{eq:dimofker}).
\end{proof}

As a direct consequence of Lemma~\ref{lem:QConS}, we recover
the result~\cite[Theorem~1.1]{TY} concerning the horizontal
connectivity of points on a hypersurface embedded in a
contact sub-Riemannian manifold of dimension $2n+1$ for $n>1$.

Moreover, we see that $\db\quasi|_{\sdistr}$ induces a line field on
$S\setminus\car(S)$ which can be oriented and extended
to the set $\car(S)$ of characteristic points to yield an
oriented singular line field on the hypersurface $S$, called the
characteristic foliation of $S$ and defined in~\cite[Definition~2.0.1]{CHT}.

\begin{remark}
  Throughout the article, we consider the distribution $\sdistr$ as
  defined on $S\setminus\car(S)$, that is, away from the set $\car(S)$
  of points $x\in S$ where $T_xS=\distr_x$. One may also regard
  $\sdistr$ as a generalised distribution given at every point of $S$
  by setting $\sdistr_{x}=T_{x}S$ for $x\in \car(S)$.

  In this viewpoint, $\sdistr$ is not a rank-varying distribution in the
  sense of vector fields because there does not exist a family of globally
  defined vector fields $Y_1,\dots,Y_m$ on $S$ such that, for all $x\in S$,
  \begin{displaymath}
    \sdistr_x=\operatorname{span}
    \left\{Y_1(x),\dots,Y_m(x)\right\}\;.
  \end{displaymath}
  Indeed, in such a case the map $x\mapsto\dim\sdistr_x$ would be lower
  semicontinuous, which is not true in our situation. Instead, the
  dimension of $\sdistr_x$ increases at singular points. This is
  typical of a distribution defined by Pfaffian equations, that is,
  the zero locus of a family of linear forms.
\end{remark}

Examples illustrating the geometry and in particular the singular
one-dimensional foliation induced on surfaces embedded in
three-dimensional contact sub-Riemannian manifolds are discussed,
among others, in~\cites{massot,BBCH,BBC}. For an
example which demonstrates the
geometry induced on hypersurfaces
in higher-dimensional contact sub-Riemannian manifolds and which
further
highlights that for $n>1$ the distribution $\sdistr$ defined away from
characteristic points becomes quasi-contact, we study a
canonical hypersurface embedded in the Heisenberg group $\Heis^2$.
\begin{example}\label{ex:surfaceinH2}
  Let $(x_1,y_1,x_2,y_2,z)$ denote Cartesian coordinates on $\R^5$
  and consider the contact form $\omega$ on $\R^5$ defined by
  \begin{equation}\label{eq:omegainH2}
    \omega=\frac{1}{2}\left(x_1\dd y_1-y_1\dd x_1\right)
    +\frac{1}{2}\left(x_2\dd y_2-y_2\dd x_2\right)-\db z\;.
  \end{equation}
  Equipping the contact structure $\distr=\ker\omega$ with the fibre
  inner product
  \begin{equation}\label{eq:H2g}
    g=\db x_1\otimes\db x_1+\db y_1\otimes\db y_1
    +\db x_2\otimes\db x_2+\db y_2\otimes\db y_2\;,
  \end{equation}
  we obtain the contact sub-Riemannian manifold $(\R^5,\distr,g)$,
  which is the Heisenberg group $\Heis^2$. We observe that our choice
  of contact form $\omega$ satisfies the normalisation
  condition~(\ref{eq:normalise}) since
  \begin{displaymath}
    \db\omega=\db x_1\wedge\db y_1
      +\db x_2\wedge\db y_2\;,
  \end{displaymath}
  which implies that
  \begin{displaymath}
    (\db\omega)^2=\db\omega\wedge\db\omega
    =2\dd x_1\wedge\db y_1\wedge\db x_2\wedge\db y_2\;,
  \end{displaymath}
  and therefore, we have
  $\left.(\db\omega)^{2}\right|_{\distr}=2\,\mathrm{vol}_{g}\,$.

  The hypersurface $S$ which we study in the Heisenberg group
  $\Heis^2$ is the one defined by $\{z=0\}$.
  It illustrates well the changes in properties of the
  distribution $\sdistr$
  for $n > 1$ compared to $n=1$ whilst
  still allowing for explicit computations and constructions.
  From
  \begin{displaymath}
    \omega\left(\frac{\pt}{\pt x_1}\right)=-\frac{y_1}{2}\;,\quad
    \omega\left(\frac{\pt}{\pt y_1}\right)=\frac{x_1}{2}\;,\quad
    \omega\left(\frac{\pt}{\pt x_2}\right)=-\frac{y_2}{2}\;,\quad
    \omega\left(\frac{\pt}{\pt y_2}\right)=\frac{x_2}{2}\;,
  \end{displaymath}
  we see that the origin of $\R^5$ is the only characteristic point of
  this hypersurface $S$.
  The distribution $\sdistr$ defined on $S\setminus\{0\}$ as
  $\distr\cap TS$ is a subbundle of corank one in the tangent
  bundle of $S\setminus\{0\}$ and can be described as the kernel of the
  one-form
  \begin{equation}\label{eq:H2quasi}
    \quasi=\frac{1}{2}\left(x_1\dd y_1-y_1\dd x_1\right)
    +\frac{1}{2}\left(x_2\dd y_2-y_2\dd x_2\right)\;,
  \end{equation}
  which is
  obtained by restricting the contact form $\omega$ to the tangent
  bundle of $S\setminus\{0\}$.
  
  To gain a better understanding of the distribution $\sdistr$,
  we find an orthonormal frame for $\sdistr$, which we later
  further work with to explicitly determine the intrinsic sub-Laplacian
  $\Delta$ on $S\setminus\{0\}$. Let $U_1,U_2,U_3,U_4$ be the vector
  fields on $S\setminus\{0\}$ defined by
  \begin{align*}
    U_1
    &=\frac{1}{\sqrt{x_1^2+y_1^2+x_2^2+y_2^2}}
      \left(x_1\frac{\pt}{\pt x_1}+y_1\frac{\pt}{\pt y_1}
      +x_2\frac{\pt}{\pt x_2}+y_2\frac{\pt}{\pt y_2}\right)\;,\\
    U_2
    &=\frac{1}{\sqrt{x_1^2+y_1^2+x_2^2+y_2^2}}
      \left(y_2\frac{\pt}{\pt x_1}+x_2\frac{\pt}{\pt y_1}
      -y_1\frac{\pt}{\pt x_2}-x_1\frac{\pt}{\pt y_2}\right)\;,\\
    U_3
    &=\frac{1}{\sqrt{x_1^2+y_1^2+x_2^2+y_2^2}}
      \left(x_2\frac{\pt}{\pt x_1}-y_2\frac{\pt}{\pt y_1}
      -x_1\frac{\pt}{\pt x_2}+y_1\frac{\pt}{\pt y_2}\right)\;,\\
    U_4
    &=\frac{1}{\sqrt{x_1^2+y_1^2+x_2^2+y_2^2}}
      \left(y_1\frac{\pt}{\pt x_1}-x_1\frac{\pt}{\pt y_1}
      +y_2\frac{\pt}{\pt x_2}-x_2\frac{\pt}{\pt y_2}\right)\;.
  \end{align*}
  Using~(\ref{eq:H2g}) and~(\ref{eq:H2quasi}), we verify that
  $(U_1,U_2,U_3)$ is an orthonormal frame for $\sdistr$, and we further
  note that
  $(U_1,U_2,U_3,U_4)$ is a frame for the tangent bundle
  of $S\setminus\{0\}$. Due to
  \begin{displaymath}
    \db\quasi=\db x_1\wedge\db y_1
      +\db x_2\wedge \db y_2\;,
  \end{displaymath}
  we obtain that
  \begin{displaymath}
    \db\quasi\left(U_1,U_2\right)=\db\quasi\left(U_1,U_3\right)=0
    \quad\text{and}\quad
    \db\quasi\left(U_2,U_3\right)
    =-1\;.
  \end{displaymath}
  It follows that
  \begin{equation}\label{eq:charfolinH2}
    \left.\ker\db\quasi\right|_\sdistr=\operatorname{span}\left\{U_{1}\right\}
    =\operatorname{span}\left\{
      x_1\frac{\pt}{\pt x_1}+y_1\frac{\pt}{\pt y_1}+
      x_2\frac{\pt}{\pt x_2}+y_2\frac{\pt}{\pt y_2}
    \right\}\;,
  \end{equation}
  and thus, consistent with Lemma~\ref{lem:QConS}, the rank three
  distribution
  $\sdistr$ is a quasi-contact structure on
  the four-dimensional hypersurface
  $S\setminus\{0\}$. According to Lemma~\ref{lem:bracketgen}, this
  implies that the distribution $\sdistr$ is
  bracket generating on 
  $S\setminus\{0\}$, which can be seen directly by noting that
  \begin{align*}
    [U_2,U_3]
    &=\frac{2}{x_1^2+y_1^2+x_2^2+y_2^2}
      \left(-y_1\frac{\pt}{\pt x_1}+x_1\frac{\pt}{\pt y_1}
      -y_2\frac{\pt}{\pt x_2}+x_2\frac{\pt}{\pt y_2}\right)\\
    &=-\frac{2U_4}{\sqrt{x_1^2+y_1^2+x_2^2+y_2^2}}\;.
  \end{align*}
  We continue our analysis for this case by determining the intrinsic
  sub-Laplacian $\Delta$ in the forthcoming
  Example~\ref{ex:lapinH2}. Moreover, in Section~\ref{sec:heis}, we
  discuss the radial part of the stochastic process with generator
  $\frac{1}{2}\Delta$, which as a result of~(\ref{eq:charfolinH2}) is
  exactly the stochastic process induced on the characteristic
  foliation of the hypersurface $S$.
\end{example}

\section{Intrinsic sub-Laplacian as limit of
  Laplace--Beltrami operators}
\label{sec:lap}
After discussing
the construction of the Laplace--Beltrami operators $\Delta^\eps$ on
the hypersurface $S$ using Riemannian approximations of the
contact sub-Riemannian manifold $(M,\distr,g)$,
we proceed with proving Theorem~\ref{thm:limitlap}.

The Riemannian approximation for $\eps>0$ to the contact
sub-Riemannian manifold
$(M,\distr,g)$ with respect to the Reeb vector field $X_0$ equips the
smooth manifold $M$ with the Riemannian metric $\overline{g}^\eps$
given by, for $\theta_0\colon T M\to \R$ the unique linear map such
that $\theta_0(X_0)=1$ and $\theta_0(X)=0$ for any vector field $X$ in
$\distr$,
\begin{displaymath}
  \overline{g}^\eps=g\oplus
  \frac{1}{\eps^2}\left(\theta_0\otimes \theta_0\right)\;.
\end{displaymath}
In particular, if $(X_1,\dots,X_{2n})$ is a positively oriented local
orthonormal frame for the distribution $\distr$ with respect to the
fibre inner product
$g$, then $(X_1,\dots,X_{2n},\eps X_0)$ is a positively oriented
orthonormal frame for the tangent bundle
$TM$ with respect to the Riemannian metric $\overline{g}^\eps$. Using
this observation, we
can establish the property for the volume form
$\vol^{\eps}=\mathrm{vol}_{\overline g^{\eps}}$ on $M$ stated below.
\begin{lemma}\label{lem:vol}
  For $\eps>0$,
  the volume forms $\vol$ and $\vol^\eps$ on the manifold $M$ are
  related by
  \begin{displaymath}
    \eps n!\,\vol^{\eps}=\vol\;.
  \end{displaymath}
\end{lemma}
\begin{proof}
  Let $(X_1,\dots,X_{2n})$ be a positively oriented local
  orthonormal frame for $\distr$. Then
  $\omega(X_i)=0$ for all $i\in\{1,\dots,2n\}$ as well as
  $\omega(X_0)=1$
  and
  $\db\omega(X_0,\cdot)=0$ together with~(\ref{eq:defnOmega}) and the
  normalisation condition~(\ref{eq:normalise})
  yield
  \begin{displaymath}
    \vol(X_{1},\dots,X_{2n},\eps X_{0})
    =n!\,\omega(\eps X_0)
    \mathrm{vol}_g(X_{1},\ldots,X_{2n})=\eps n!\;.
  \end{displaymath}
  On the other hand, we have, by construction,
  \begin{displaymath}
    \vol^{\eps}(X_{1},\dots,X_{2n},\eps X_{0})=1\;,
  \end{displaymath}
  which implies the claimed result.
\end{proof}

Similarly to Definition~\ref{defn:sRnormal} for the sub-Riemannian
normal vector field
$N$ to $S\setminus\car(S)$ in the contact
sub-Riemannian manifold $(M,\distr,g)$, we define the Riemannian
normal vector field $N^\eps$ for $\eps>0$ to the hypersurface $S$
embedded in
the Riemannian manifold $(M,\overline{g}^\eps)$
of dimension $2n+1$.
\begin{definition}\label{defn:Rnormal}
  The Riemannian normal vector field $N^\eps$ along the hypersurface
  $S$ embedded
  in the Riemannian manifold $(M,\overline{g}^\eps)$ is the
  unit-length vector field along $S$, that is,
  \begin{displaymath}
    \overline{g}^\eps(N^\eps,N^\eps)=1\;,
  \end{displaymath}
  such that, for any vector field $Z$ on $S$,
  \begin{displaymath}
    \overline{g}^\eps(N^\eps,Z)=0\;,
  \end{displaymath}
  and, for any positively oriented local orthonormal frame
  $(Z_{1},\dots,Z_{2n})$ for $S\setminus\car(S)$, the frame
  $(N^\eps,Z_{1},\dots,Z_{2n})$ for $M$ is a positively oriented.
\end{definition}

The next result states that as $\eps\to 0$ the Riemannian
normal vector fields $N^\eps$ converge uniformly on compact subsets of
$S\setminus\car(S)$ to the sub-Riemannian normal vector field $N$.
\begin{lemma}\label{lem:normal}
  Uniformly on compact subsets of
  $S\setminus\car(S)$, we have
  \begin{displaymath}
    N^\eps\to N\quad\text{as}\quad\eps\to 0\;.
  \end{displaymath}
\end{lemma}
\begin{proof}
  We use that the hypersurface $S$ is locally given
  as the zero set
  of some smooth function $u\in C^\infty(M)$ with $\db u\not=0$ on $S$,
  and we fix a local orthonormal frame $(X_1,\dots,X_{2n})$ for
  the contact structure $\distr$ with respect to the fibre inner
  product $g$.

  Since $x\in S$
  is a characteristic point of the hypersurface $S$ if
  the tangent space $T_x S$ coincides with $\distr_x$, that is, if
  \begin{displaymath}
    \left(X_iu\right)(x)=0
    \quad\text{for all}\enspace
    i\in\{1,2,\dots,2n\}\;,
  \end{displaymath}
  we have
  \begin{equation}\label{eq:char4char}
    \car(S)=
    \left\{x\in S:\sum_{i=1}^{2n}\left((X_i u)(x)\right)^2=0\right\}\;.
  \end{equation}

  In terms of the local orthonormal frame $(X_1,\dots,X_{2n})$ for
  the distribution $\distr$ and with $\sigma=0$ or $\sigma=1$
  depending on the orientation of $S$,
  the sub-Riemannian normal vector field $N$
  along $S\setminus\car(S)$ can be written as
  \begin{equation}\label{eq:Ninframe}
    N=(-1)^\sigma\dfrac{\sum_{i=1}^{2n}(X_i u)X_i}
    {\sqrt{\sum_{i=1}^{2n}(X_i u)^2}}\;,
  \end{equation}
  due to the following reasoning. The expression~(\ref{eq:Ninframe})
  is well-defined away from the set of characteristic points
  as a result of~(\ref{eq:char4char}). Moreover, the conditions in
  Definition~\ref{defn:sRnormal} are satisfied
  because of $(X_1,\dots,X_{2n})$ being
  an orthonormal frame for $\distr$ and since, for any vector
  field $Y$ in the distribution $\sdistr=\distr\cap TS$ on
  $S\setminus\car(S)$,
  \begin{displaymath}
    g(N,Y)=(-1)^\sigma \dfrac{\sum_{i=1}^{2n}(X_i u)g(X_i,Y)}
    {\sqrt{\sum_{i=1}^{2n}(X_i u)^2}}
    =(-1)^\sigma \dfrac{Yu}
    {\sqrt{\sum_{i=1}^{2n}(X_i u)^2}}=0\;.
  \end{displaymath}
  Similarly, we verify that the Riemannian normal
  vector field $N^\eps$ for $\eps>0$ to the hypersurface $S$
  can be expressed as
  \begin{equation}\label{eq:Nepsinframe}
    N^\eps=(-1)^\sigma \dfrac{\sum_{i=1}^{2n}(X_i u)X_i+\eps^2(X_0u)X_0}
    {\sqrt{\sum_{i=1}^{2n}(X_i u)^2+\eps^2(X_0u)^2}}\;.
  \end{equation}
  The claimed result then follows from~(\ref{eq:Ninframe}),
  (\ref{eq:Nepsinframe}) and $u\in C^\infty(M)$.
\end{proof}

The Riemannian volume form $\mu^\eps$ induced on the hypersurface $S$
embedded in the Riemannian manifold $(M,\overline{g}^\eps)$ is
given on $S$ by
\begin{equation}\label{eq:defnepsvol}
  \mu^\eps=\iota_{N^\eps}\vol^\eps\;,
\end{equation}
and the Riemannian gradient $\nabla_S^\eps f$ of a smooth function
$f\colon S\to\R$ is uniquely characterised by requiring that, for any
vector field $Z$ on $S$,
\begin{displaymath}
  \overline{g}^\eps(\nabla_S^\eps f,Z)=\db f(Z)\;.
\end{displaymath}
The Laplace--Beltrami operator $\Delta^\eps$ on the
Riemannian manifold $(S,i^\ast\overline{g}^\eps)$,
where $i\colon S\to M$ is the inclusion map,
is then defined
by, for a smooth function $f\colon S \to \R$,
\begin{displaymath}
  \Delta^\eps f=\operatorname{div}_{\mu^\eps}\left(\nabla_S^\eps f\right)\;.
\end{displaymath}

The following result is crucial in proving the convergence of the
Laplace--Beltrami operators $\Delta^\eps$ as $\eps\to 0$ to the intrinsic
sub-Laplacian $\Delta$.
\begin{lemma}\label{lem:surfvol}
  Uniformly on compact subsets of
  $S\setminus\car(S)$, we have
  \begin{displaymath}
    \eps n!\,\mu^\eps\to \mu\quad\text{as}\quad\eps\to 0\;.
  \end{displaymath}
\end{lemma}
\begin{proof}
  This is a direct consequence of Definition~\ref{defn:hypervol}, the
  identity~(\ref{eq:defnepsvol}), Lemma~\ref{lem:vol} and
  Lemma~\ref{lem:normal}.
\end{proof}

We are finally in a position to prove Theorem~\ref{thm:limitlap}. Note
that as a result of Lemma~\ref{lem:bracketgen} and
Lemma~\ref{lem:QConS}, the intrinsic sub-Laplacian $\Delta$ is indeed
a hypoelliptic operator on $S\setminus\car(S)$ as long as $n\geq 2$.
\begin{proof}[Proof of Theorem~\ref{thm:limitlap}]
  Choose a local orthonormal frame $(Y_1,\dots, Y_{2n-1})$ for
  $\sdistr$. From the observation~(\ref{eq:horgrad}), it
  follows that the intrinsic sub-Laplacian $\Delta$ on
  $S\setminus\car(S)$ can be written as
  \begin{equation}\label{eq:subLap}
    \Delta=\sum_{i=1}^{2n-1}
    \left(Y_i^2+
      \left(\operatorname{div}_{\mu}Y_i\right)Y_i\right)\;.
  \end{equation}
  We now aim to extend the local orthonormal frame $(Y_1,\dots, Y_{2n-1})$ for
  $\sdistr$ to a local orthonormal frame for the tangent bundle
  of $S\setminus\car(S)$ with respect to the Riemannian metric
  $\overline{g}^\eps$. To this end, we again use that the
  hypersurface $S$ is locally given as the zero set
  of some smooth function $u\in C^\infty(M)$ with $\db u\not=0$ on $S$
  and we consider the vector field $Z$ on
  $S\setminus\car(S)$ given by
  \begin{displaymath}
    Z=X_0-\frac{X_0u}{Nu}N\;.
  \end{displaymath}
  This vector field $Z$ can be seen as the projection of the Reeb
  vector field $X_0$ on $M$ onto the hypersurface
  $S\setminus\car(S)$. Using (\ref{defn:sRnormal1}) and
  (\ref{eq:Rapprox}), we compute, for $\eps>0$,
  \begin{equation}\label{eq:normproj}
    \overline{g}^\eps(Z,Z)=\frac{1}{\eps^2}+\frac{(X_0u)^2}{(Nu)^2}>0\;,
  \end{equation}
  which implies that we can define a vector field $Z^\eps$ on
  $S\setminus\car(S)$ by setting
  \begin{displaymath}
    Z^\eps=\frac{Z}{\sqrt{\overline{g}^\eps(Z,Z)}}\;.
  \end{displaymath}
  Since both the Reeb vector field $X_0$ and the sub-Riemannian normal
  vector field $N$ are orthogonal with respect to the Riemannian
  metric $\overline{g}^\eps$ to any vector field in $\sdistr$, it
  follows that $(Y_1,\dots, Y_{2n-1},Z^\eps)$ is a local orthonormal
  frame for the tangent bundle of $S\setminus\car(S)$ with respect to
  the Riemannian metric $\overline{g}^\eps$. Similarly as above, we
  can then express the
  Laplace--Beltrami operator $\Delta^\eps$ on $S\setminus\car(S)$ as
  \begin{equation}\label{eq:LBeps}
    \Delta^\eps=\sum_{i=1}^{2n-1}
    \left(Y_i^2+
      \left(\operatorname{div}_{\mu^\eps}Y_i\right)Y_i\right)
    +\left(Z^\eps\right)^2
    +\left(\operatorname{div}_{\mu^\eps}Z^\eps\right)Z^\eps\;.
  \end{equation}
  From~(\ref{eq:normproj}), we deduce
  \begin{displaymath}
    \frac{1}{\sqrt{\overline{g}^\eps(Z,Z)}}\leq \eps\;,
  \end{displaymath}
  which shows that for any
  smooth function $f\in C_c^\infty(S\setminus\car(S))$
  compactly supported in $S\setminus\car(S)$, we have, as $\eps\to 0$
  and uniformly on $S\setminus\car(S)$,
  \begin{equation}\label{eq:Zepsconv}
    \left(Z^\eps\right)^2f\to 0
    \quad\text{and}\quad
    Z^\eps f\to 0\;.
  \end{equation}
  Therefore, it remains to analyse the divergence terms
  in the expression~(\ref{eq:LBeps}) for the Laplace--Beltrami
  operator $\Delta^\eps$. Working in a local coordinate chart
  $(x_1,\dots,x_{2n})$ for the hypersurface $S\setminus\car(S)$,
  we let $\rho$ and $\rho^\eps$ denote the local coefficient of $\mu$
  and $\mu^\eps$, respectively, and
  we
  use Lemma~\ref{lem:surfvol} as well as the uniform convergence of the
  derivatives on compacts, which can be established similarly, to
  argue that, for all $i\in\{1,\dots,2n\}$ and uniformly on compact
  subsets of $S\setminus\car(S)$,
  \begin{align}\label{eq:divYconv}
    \begin{aligned}
      \lim_{\eps\to 0}\operatorname{div}_{\mu^\eps}Y_i
      =\lim_{\eps\to 0}\sum_{j=1}^{2n}
      \frac{1}{\eps n!\,\rho^\eps}\frac{\pt}{\pt x_j}
      \left(\eps n!\,\rho^\eps Y_{i,j}\right)
      =\sum_{j=1}^{2n}\frac{1}{\rho}\frac{\pt}{\pt x_j}
      \left(\rho Y_{i,j}\right)
      =\operatorname{div}_{\mu}Y_i\;.
    \end{aligned}
  \end{align}
  Similarly, we conclude that, as $\eps\to 0$ and uniformly on compact
  subsets of $S\setminus\car(S)$,
  \begin{displaymath}
    \operatorname{div}_{\mu^\eps}Z\to\operatorname{div}_{\mu}Z\;.
  \end{displaymath}
  Hence, as a consequence of
  \begin{displaymath}
    \operatorname{div}_{\mu^\eps}Z^\eps
    =\frac{\operatorname{div}_{\mu^\eps}Z}{\sqrt{\overline{g}^\eps(Z,Z)}}
    +\overline{g}^\eps\left(
      \nabla_S^\eps\left(\frac{1}{\sqrt{\overline{g}^\eps(Z,Z)}}\right),Z
    \right)
  \end{displaymath}
  and since
  \begin{displaymath}
    \overline{g}^\eps\left(
      \nabla_S^\eps\left(\frac{1}{\sqrt{\overline{g}^\eps(Z,Z)}}\right),Z
    \right)
    =Z^\eps\left(\frac{1}{\sqrt{\overline{g}^\eps(Z,Z)}}\right)
    \overline{g}^\eps\left(Z^\eps,Z\right)
    =Z\left(\frac{1}{\sqrt{\overline{g}^\eps(Z,Z)}}\right)\;,
  \end{displaymath}
  we obtain that, as $\eps\to 0$ and uniformly on compact
  subsets of $S\setminus\car(S)$,
  \begin{displaymath}
    \operatorname{div}_{\mu^\eps}Z^\eps\to 0\;.
  \end{displaymath}
  Together with~(\ref{eq:Zepsconv}) and (\ref{eq:divYconv}), the
  claimed result then follows from~(\ref{eq:subLap}) and (\ref{eq:LBeps}).
\end{proof}

As a first illustration of the general strategy laid out for constructing
the intrinsic sub-Laplacian $\Delta$,
we return to our analysis for the hypersurface $\{z=0\}$ in the
Heisenberg group
$\Heis^2$ started in Example~\ref{ex:surfaceinH2}
and we demonstrate how to
derive an explicit expression for the intrinsic sub-Laplacian $\Delta$
on $\{z=0\}$ away from the unique characteristic point at the origin.
\begin{example}\label{ex:lapinH2}
  As discussed in Example~\ref{ex:surfaceinH2}, the quasi-contact
  structure
  $\sdistr$ on $S\setminus\{0\}$ admits the orthonormal frame
  $(U_1,U_2,U_3)$ with respect to $g$. It
  follows that the horizontal gradient
  $\nabla_S f$ of a smooth function $f\colon S\setminus\{0\}\to\R$ can
  be expressed as
  \begin{displaymath}
    \nabla_S f=(U_1f)U_1+(U_2f)U_2+(U_3f)U_3\;.
  \end{displaymath}
  It remains to determine the volume form $\mu$ on
  the hypersurface $S\setminus\{0\}$
  and to compute the
  divergence of the vector fields $U_1,U_2$ and $U_3$ with respect
  to $\mu$.

  The volume form $\vol$ on $\R^5$ defined by~(\ref{eq:defnOmega})
  in terms of the contact form $\omega$ in~(\ref{eq:omegainH2}) is
  given by
  \begin{displaymath}
    \vol=\omega\wedge\left(\db\omega\right)^2
    =-2\dd x_1\wedge\db y_1\wedge\db x_2\wedge\db y_2\wedge \db z\;,
  \end{displaymath}
  and the sub-Riemannian normal vector field
  $N$ along the hypersurface $S\setminus\{0\}$ in $\Heis^2$
  characterised
  by~(\ref{defn:sRnormal1}) as well as (\ref{defn:sRnormal2})
  and compatible with the orientations on $M$ and $S$
  can be written as
  \begin{displaymath}
    N=U_4-\frac{\sqrt{x_1^2+y_1^2+x_2^2+y_2^2}}{2}\frac{\pt}{\pt z}\;.
  \end{displaymath}
  It follows that defining the volume form $\mu$
  on $S\setminus\{0\}$ as $\iota_N\vol$ yields
  \begin{displaymath}
    \mu=\sqrt{x_1^2+y_1^2+x_2^2+y_2^2}\,
    \dd x_1\wedge\db y_1\wedge\db x_2\wedge\db y_2\;.
  \end{displaymath}
  This implies that
  \begin{displaymath}
    \operatorname{div}_\mu U_1
    =\frac{4}{\sqrt{x_1^2+y_1^2+x_2^2+y_2^2}}
    \quad\text{and}\quad
    \operatorname{div}_\mu U_2=
    \operatorname{div}_\mu U_3=0\;.
  \end{displaymath}
  Thus, the intrinsic sub-Laplacian $\Delta$ on
  the hypersurface $S\setminus\{0\}$ in the Heisenberg group $\Heis^2$
  can be expressed as
  \begin{equation}\label{eq:DeltainH2}
    \Delta=
    U_1^2+U_2^2+U_3^3
    +\frac{4U_1}{\sqrt{x_1^2+y_1^2+x_2^2+y_2^2}}\;.
  \end{equation}

  Due to the quasi-contact structure $\sdistr$
  on $S\setminus\{0\}$ being bracket generating, the
  intrinsic sub-Laplacian
  $\Delta$ is hypoelliptic, see
  H\"{o}rmander~\cite{hormander}. This illustrates a crucial change in
  property of the intrinsic sub-Laplacian for $n>1$ compared to the
  case $n=1$.
  As seen in~\cite{BBCH}, the operator $\Delta$ on
  surfaces in three-dimensional contact sub-Riemannian manifolds is never
  hypoelliptic as a result of $\sdistr$ being a line field in that setting.

  We close by highlighting that, as discussed in more details in the
  forthcoming analysis in
  Section~\ref{sec:model}, thanks to the drift term
  in~(\ref{eq:DeltainH2}), the intrinsic sub-Laplacian $\Delta$ on
  $S\setminus\{0\}$ is stochastically complete, and in particular,
  the stochastic process with
  generator $\frac{1}{2}\Delta$ on $S\setminus\{0\}$ almost surely
  does not hit the unique characteristic point at the origin.
\end{example}

\section{Canonical hypersurfaces in contact
  sub-Riemannian model spaces}
\label{sec:model}
We consider canonical hypersurfaces in contact sub-Riemannian
model spaces which extend the family of model cases given
in~\cite[Theorem~1.5]{BBCH}. Choosing suitable coordinates, we establish
Proposition~\ref{propn:vol} by
explicitly computing the volume form $\mu$ induced on the hypersurface
away from characteristic points. This in turn allows us to
prove Theorem~\ref{thm:radial}, which characterises
the radial part of the stochastic process with generator
$\frac{1}{2}\Delta$ and which implies that in these model
cases analysed the intrinsic sub-Laplacian
$\Delta$ defined on the hypersurface away from characteristic points
is stochastically complete, whilst the induced geometry is not
geodesically complete.

We first study $\R^{2n}$ suitably embedded in the Heisenberg group
$\Heis^n$ for $n\geq 1$, which pushes the analysis from
Example~\ref{ex:surfaceinH2} and Example~\ref{ex:lapinH2} to all
possible
dimensions, with the exception we do not provide a full
expression for the intrinsic sub-Laplacian. Instead, we
restrict our attention to its radial contribution.

We then proceed by considering the sphere $S^{2n}$ embedded in
$S^{2n+1}$ equipped with the standard sub-Riemannian contact structure
subject to an additional parameter $k>0$,
and the hyperboloid $\Sinads^{2n}$ embedded in $\ads^{2n+1}$ equipped
with the standard
sub-Riemannian contact structure subject to an additional parameter $k>0$.

\subsection{\texorpdfstring{$\R^{2n}$}{2n-space}
  embedded in \texorpdfstring{$\R^{2n+1}$}{(2n+1)-space}}
\label{sec:heis}
Let $(x_1,\dots,x_{2n},x_{2n+1})$ be Cartesian coordinates on
$\R^{2n+1}$. Use the contact form $\omega$ on $\R^{2n+1}$ given by
\begin{equation}\label{eq:contactinH}
  \omega=\frac{1}{2}\sum_{m=1}^n
  \left(x_{2m-1}\db x_{2m}-x_{2m}\db x_{2m-1}\right)-\db x_{2n+1}
\end{equation}
to define the contact structure $\distr=\ker\omega$ on $\R^{2n+1}$.
As fibre inner product $g$ on $\distr$, we take
\begin{equation}\label{eq:metricinH}
  g=\sum_{i=1}^{2n}\db x_i\otimes\db x_i\;.
\end{equation}
This is the unique fibre inner product on the distribution
  $\distr$ such that the vector fields, for $m\in\{1,\dots,n\}$,
  \begin{displaymath}
    X_{2m-1}=\frac{\pt}{\pt x_{2m-1}}-\frac{x_{2m}}{2}\frac{\pt}{\pt
      x_{2n+1}}\;,\quad
    X_{2m}=\frac{\pt}{\pt x_{2m}}+\frac{x_{2m-1}}{2}\frac{\pt}{\pt
      x_{2n+1}}
  \end{displaymath}
  form an orthonormal frame $(X_1,\dots,X_{2n})$ for $\distr$.

We obtain the contact sub-Riemannian manifold $(\R^{2n+1},\distr,g)$,
which is referred to as Heisenberg group $\Heis^n$. The contact form
$\omega$ given in~(\ref{eq:contactinH}) satisfies the imposed
normalisation condition~(\ref{eq:normalise}) because
\begin{displaymath}
  \db\omega=\sum_{m=1}^n\db x_{2m-1}\wedge \db x_{2m}
\end{displaymath}
gives rise to
\begin{displaymath}
  \left(\db\omega\right)^n
  =n!\bigwedge_{i=1}^{2n}\db x_i\;,
\end{displaymath}
whilst~(\ref{eq:metricinH}) implies that
\begin{displaymath}
  \mathrm{vol}_g=\bigwedge_{i=1}^{2n}\db x_i\;.
\end{displaymath}
We further deduce that the volume form $\vol$ on
$\R^{2n+1}$ defined by~(\ref{eq:defnOmega}) can be expressed as
\begin{equation}\label{eq:amvolinH}
  \vol=- n!\bigwedge_{i=1}^{2n+1}\db x_i\;.
\end{equation}

The hypersurface $S$ in $\Heis^n$ which we study closer is the one
given by $\{x_{2n+1}=0\}$. Since, for $m\in\{1,\dots,n\}$, we have
\begin{displaymath}
  \omega\left(\frac{\pt}{\pt x_{2m-1}}\right)
  =-\frac{x_{2m}}{2}
  \quad\text{and}\quad
  \omega\left(\frac{\pt}{\pt x_{2m}}\right)
  =\frac{x_{2m-1}}{2}\;,
\end{displaymath}
the set $\car(S)$ of characteristic points contains only the origin
of $\R^{2n+1}$. Moreover, the quasi-contact form $\quasi$ induced on
$S\setminus\car(S)$ by the contact form $\omega$ on $\R^{2n+1}$ is
\begin{displaymath}
  \quasi=\frac{1}{2}\sum_{m=1}^n
  \left(x_{2m-1}\db x_{2m}-x_{2m}\db x_{2m-1}\right)\;.
\end{displaymath}
Due to the kernel $\left.\ker\db\quasi\right|_{\sdistr}$ with
$\sdistr=\ker\quasi$ being guaranteed to be one-dimensional by
Lemma~\ref{lem:QConS}, we can verify directly that
\begin{equation}\label{eq:foliainH}
  \left.\ker\db\quasi\right|_{\sdistr}
  =\operatorname{span}
  \left\{
    \sum_{i=1}^{2n}x_i\frac{\pt}{\pt x_i}\right\}\;.
\end{equation}

The lemma stated below provides an expression for the sub-Riemannian
normal vector field $N$
along $S\setminus\car(S)$ in $\Heis^n$, which we prove in detail as a
similar approach can be used to confirm the expressions for the
sub-Riemannian normal
vector fields in Section~\ref{sec:sphere} and
Section~\ref{sec:hyper}.
\begin{lemma}\label{lem:NinH}
  The sub-Riemannian normal vector field $N$ along $S\setminus\car(S)$
  in $\Heis^n$ is given by
  \begin{displaymath}
    N=\frac{1}{\sqrt{\sum_{i=1}^{2n}x_i^2}}
    \left(\sum_{m=1}^n\left(        
       x_{2m}\frac{\pt}{\pt x_{2m-1}}-x_{2m-1}\frac{\pt}{\pt
         x_{2m}}\right)
     -\frac{1}{2}\sum_{i=1}^{2n}x_i^2\frac{\pt}{\pt x_{2n+1}}
    \right)\;.
  \end{displaymath}  
\end{lemma}
\begin{proof}
  Since the vector field $N$ is well-defined along the hypersurface
  $S$ away from the unique characteristic point at the origin of
  $\R^{2n+1}$, it remains to check that $N$ satisfies the defining
  properties~(\ref{defn:sRnormal1}) as well as (\ref{defn:sRnormal2})
  and that it is compatible with the orientations.
  
  From the expressions for the contact form $\omega$
  in~(\ref{eq:contactinH}) and the fibre inner product $g$ in
  (\ref{eq:metricinH}), it follows that
  \begin{displaymath}
    \omega(N)=0
    \quad\text{and}\quad
    g(N,N)=1\;.
  \end{displaymath}
  Furthermore, using that any vector field $Y$ in the
  distribution $\sdistr$ satisfies $\quasi(Y)=0$, we deduce
  \begin{displaymath}
    g(N,Y)=-\frac{2\quasi(Y)}{\sqrt{\sum_{i=1}^{2n}x_i^2}}=0\;.
  \end{displaymath}
  Finally, we obtain from~(\ref{eq:amvolinH}) that $\iota_N\vol$ is
  positive on $S\setminus\car(S)$,
  which shows that $N$ is indeed the sub-Riemannian normal vector field along
  $S\setminus\car(S)$ in $\Heis^n$ according to
  Definition~\ref{defn:sRnormal}.
\end{proof}

Using the expression~(\ref{eq:amvolinH}) for the volume form $\vol$
on $\R^{2n+1}$ as well as Lemma~\ref{lem:NinH}, we compute that the volume
form $\mu$ defined on $S\setminus\car(S)$ as $\iota_N\vol$ is given by
\begin{equation}\label{eq:volinH}
  \mu=\frac{n!}{2}\sqrt{\sum_{i=1}^{2n}x_i^2}
  \bigwedge_{i=1}^{2n}\db x_i\;.
\end{equation}
At this point, it is convenient to
change from Cartesian coordinates $(x_1,\dots,x_{2n})$
for $S\setminus\car(S)$
to spherical
coordinates $(r,\varphi_1,\dots,\varphi_{2n-1})$
with $r>0$, $\varphi_1,\dots,\varphi_{2n-2}\in[0,\pi]$ and
$\varphi_{2n-1}\in[0,2\pi)$, where
\begin{align*}
  x_i&=r\cos(\varphi_i)\prod_{l=1}^{i-1}\sin(\varphi_l)
       \quad\text{for }i\in\{1,\dots,2n-1\}\;,\\
  x_{2n}&=r\prod_{l=1}^{2n-1}\sin(\varphi_l)\;.
\end{align*}
By means of induction over $n\geq 1$, it can be shown explicitly that
the determinant of the
associated Jacobian matrix $J_{2n}$ equals
\begin{displaymath}
  \det J_{2n}=r^{2n-1}\prod_{i=1}^{2n-2}\left(\sin(\varphi_i)\right)^{2n-i-1}\;.
\end{displaymath}
Since we further know that
\begin{displaymath}
  \bigwedge_{i=1}^{2n}\db x_i
  =\det J_{2n}\dd r\wedge\bigwedge_{i=1}^{2n-1}\db\varphi_i
  \quad\text{and}\quad
  \sqrt{\sum_{i=1}^{2n}x_i^2}=r\;,
\end{displaymath}
the expression for the volume form $\mu$ on $S\setminus\car(S)$ in
$\Heis^n$ stated in Proposition~\ref{propn:vol} follows
from~(\ref{eq:volinH}).

We close by analysing the radial part of the stochastic process with
generator $\frac{1}{2}\Delta$ on $S\setminus\car(S)$.
Using~(\ref{eq:metricinH}), (\ref{eq:foliainH}) and
\begin{displaymath}
  \frac{\pt}{\pt r}
  =\sum_{i=1}^{2n}\frac{\pt x_i}{\pt r}\frac{\pt}{\pt x_i}
  =\frac{1}{\sqrt{\sum_{i=1}^{2n}x_i^2}}
    \sum_{i=1}^{2n}x_i\frac{\pt}{\pt x_i}\;,
\end{displaymath}
we obtain that
\begin{displaymath}
  \left.\ker\db\quasi\right|_{\sdistr}
  =\operatorname{span}
  \left\{\frac{\pt}{\pt r}\right\}
  \quad\text{as well as}\quad
  g\left(\frac{\pt}{\pt r},\frac{\pt}{\pt r}\right)=1\;.
\end{displaymath}
Thus, the vector field
$R=\frac{\pt}{\pt r}$ defined on $S\setminus\car(S)$ is a unit-length
representative of the characteristic foliation induced on the hypersurface $S$
by the contact structure $\distr$. We compute
\begin{displaymath}
  \operatorname{div}_\mu\left(R\right)=\frac{2n}{r}\;,
\end{displaymath}
which implies that the radial part of the stochastic process with
generator $\frac{1}{2}\Delta$ on $S\setminus\car(S)$ is the
one-dimensional diffusion process on $(0,\infty)$ with generator
\begin{displaymath}
  \frac{1}{2}\frac{\pt^2}{\pt r^2}+\frac{n}{r}\frac{\pt}{\pt r}\;.
\end{displaymath}
This indeed gives rise to a Bessel process of order $2n+1$, which proves
the first part of Theorem~\ref{thm:radial}. Since a Bessel process of order
$2n+1$ for all $n\geq 1$ almost surely neither hits the origin nor
explodes in finite time, it follows that the intrinsic sub-Laplacian
$\Delta$ on $S\setminus\car(S)$
is stochastically complete. On the
other hand, the geometry induced on the hypersurface
$S\setminus\car(S)$ is not
geodesically complete because rays emanating from the characteristic
point and
parameterised by arc length are geodesics which cannot be extended
indefinitely towards the characteristic point not included in
the underlying space.

\begin{remark}
  Taking $n=1$, we recover the analysis for the plane $\{x_3=0\}$ in
  the Heisenberg group $\Heis^1$ which arises
  from~\cite[Section~4.1]{BBCH} by considering $a=0$, with the
  contact forms differing by a sign as a result of the normalisation
  conditions differing by a sign.
\end{remark}

\subsection{\texorpdfstring{$S^{2n}$}{2n-sphere}
  embedded in \texorpdfstring{$S^{2n+1}$}{(2n+1)-sphere}}
\label{sec:sphere}

In terms of Cartesian coordinates $(x_1,\dots,x_{2n+2})$ for $\R^{2n+2}$, we
take $S^{2n+1}\subset \R^{2n+2}$ to be
\begin{displaymath}
  S^{2n+1}=\left\{(x_1,\dots,x_{2n+2})\in\R^{2n+2}:
    \sum_{i=1}^{2n+2}x_i^2=1
  \right\}\;.
\end{displaymath}
Fix $k\in\R$ positive and consider the contact form $\omega$ on the
sphere $S^{2n+1}$ given by
\begin{equation}\label{eq:cf4sphere}
  \omega=\frac{1}{2k^2}\sum_{m=1}^{n+1}\left(
    x_{2m-1}\db x_{2m}-x_{2m}\db x_{2m-1}
  \right)\;.
\end{equation}
We further equip the contact structure $\distr=\ker\omega$  on
$S^{2n+1}$ with a smooth fibre inner product $g$ obtained by
restricting a positive constant multiple of the Euclidean inner
product $\langle\cdot,\cdot\rangle$ on $\R^{2n+2}$. More precisely, we set,
for vector fields $X_1$ and $X_2$ in $\distr$,
\begin{displaymath}
  g(X_1,X_2)=\frac{1}{k^2}\langle X_1,X_2\rangle\;.
\end{displaymath}
This construction gives rise to the standard sub-Riemannian contact
structure $(\distr,g)$ on $S^{2n+1}$ with an additional parameter
$k>0$ which mimics the introduction of an additional scalar
in~\cite[Section~5.1]{BBCH} and which later allows us to recover all
Legendre processes of order $2n+1$.

It follows from the following considerations that the
choice~(\ref{eq:cf4sphere}) of contact form $\omega$ is line with the
normalisation condition~(\ref{eq:normalise}).
We compute
\begin{displaymath}
  \db\omega=\frac{1}{k^2}\sum_{m=1}^{n+1}\db x_{2m-1}\wedge \db x_{2m}
\end{displaymath}
as well as
\begin{displaymath}
  \left(\db\omega\right)^n
  =\frac{n!}{k^{2n}}
  \sum_{m=1}^{n+1}\bigwedge_{\substack{l=1 \\ l\not=2m-1,2m}}^{2n+2}\db x_l\;,
\end{displaymath}
which implies that the volume form $\vol$ on $S^{2n+1}$ defined
by~(\ref{eq:defnOmega}) takes the form
\begin{equation}\label{eq:Shelp1}
  \vol=\frac{n!}{2k^{2n+2}}
  \sum_{i=1}^{2n+2}(-1)^{i-1}x_i
    \bigwedge_{\substack{l=1 \\ l\not= i}}^{2n+2}\db x_l\;.
\end{equation}
On the other hand, the volume form on Euclidean space $\R^{2n+2}$ with
respect to the inner product $\frac{1}{k^2}\langle\cdot,\cdot\rangle$
can be expressed as
\begin{displaymath}
  \frac{1}{k^{2n+2}}\bigwedge_{i=1}^{2n+2}\db x_i\;.
\end{displaymath}
Since $(kx_1,\dots,kx_{2n+2})$ is the unit normal vector at
$(x_1,\dots,x_{2n+2})\in S^{2n+1}$ for the inner product
$\frac{1}{k^2}\langle\cdot,\cdot\rangle$, the above volume form on
$\R^{2n+2}$ induces the volume form $\operatorname{vol}_k^{S^{2n+1}}$
on the sphere $S^{2n+1}$ given by
\begin{equation}\label{eq:Shelp2}
  \operatorname{vol}_k^{S^{2n+1}}
  =\frac{1}{k^{2n+1}}\sum_{i=1}^{2n+2}(-1)^{i-1}x_i
  \bigwedge_{\substack{l=1 \\ l\not= i}}^{2n+2}\db x_l\;.
\end{equation}
To restrict the volume form $\operatorname{vol}_k^{S^{2n+1}}$ to the
contact structure $\distr$, we use the vector field $\hat{X}_0$ defined by
\begin{equation}\label{eq:Shelp3}
  \hat{X}_0=k\sum_{m=1}^{n+1}\left(
    x_{2m-1}\frac{\pt}{\pt x_{2m}}-x_{2m}\frac{\pt}{\pt x_{2m-1}}
  \right)\;,
\end{equation}
which is the positive constant multiple of the Reeb vector field $X_0$
such that
\begin{displaymath}
  \frac{1}{k^2}\langle\hat{X}_0,\hat{X}_0\rangle=1\;.
\end{displaymath}
To establish that the contact form $\omega$ indeed satisfies the
normalisation condition~(\ref{eq:normalise}),
it remains to observe that~(\ref{eq:Shelp1}), (\ref{eq:Shelp2})
and (\ref{eq:Shelp3}) imply
\begin{displaymath}
  \left.\left(\db\omega\right)^n\right|_\distr
  =\frac{1}{\omega(\hat{X}_0)}\iota_{\hat{X}_0}\vol
  =n!\,\iota_{\hat{X}_0}\operatorname{vol}_k^{S^{2n+1}}
  =n!\operatorname{vol}_g\;.
\end{displaymath}

In the contact sub-Riemannian manifold $(S^{2n+1},\distr,g)$ with
parameter $k>0$, we study the hypersurface $S$
given by $\{x_{2n+2}=0\}$. Phrased differently, we choose
\begin{displaymath}
  S=\left\{(x_1,\dots,x_{2n+2})\in S^{2n+1}:
    \sum_{i=1}^{2n+1}x_i^2=1
  \right\}\;,
\end{displaymath}
which shows that the hypersurface $S$ can be identified with the
sphere $S^{2n}$. The set $\car(S)$ of characteristic points
contains exactly the two poles given by
\begin{displaymath}
  x_1=x_2=\dots=x_{2n}=0
  \quad\text{and}\quad
  x_{2n+1}=\pm 1\;,
\end{displaymath}
and the contact form $\omega$ on $S^{2n+1}$ induces the quasi-contact
form $\quasi$ on $S\setminus\car(S)$ defined by
\begin{equation}\label{eq:quasiinS}
  \quasi=\frac{1}{2k^2}\sum_{m=1}^{n}\left(
    x_{2m-1}\db x_{2m}-x_{2m}\db x_{2m-1}
  \right)\;.
\end{equation}

Using the approach demonstrated in the proof of Lemma~\ref{lem:NinH},
one verifies that the sub-Riemannian
normal vector field $N$ along $S\setminus\car(S)$ in
$S^{2n+1}$ can be expressed as
\begin{displaymath}
  N=\frac{k}{\sqrt{\sum_{i=1}^{2n}x_i^2}}
  \left(\sum_{m=1}^n\left(
      x_{2m}x_{2n+1}\frac{\pt}{\pt x_{2m-1}}
      -x_{2m-1}x_{2n+1}\frac{\pt}{\pt x_{2m}}\right)
    +\sum_{i=1}^{2n}x_i^2\frac{\pt}{\pt x_{2n+2}}
  \right)\;.
\end{displaymath}
This allows us to prove the next result.

\begin{lemma}\label{lem:volinsphere}
  The volume form $\mu$ defined on $S\setminus\car(S)$ as
  $\iota_N\vol$ is given by
  \begin{displaymath}
    \mu=\frac{n!}{2k^{2n+1}}\sqrt{\sum_{i=1}^{2n}x_i^2}
    \sum_{i=1}^{2n+1}(-1)^{i-1}x_i
    \bigwedge_{\substack{l=1 \\ l\not= i}}^{2n+1}\db x_l\;.
  \end{displaymath}
\end{lemma}
\begin{proof}
  Since the hypersurface $S$ is defined by $\{x_{2n+2}=0\}$ in
  $S^{2n+1}$, the interior product $\iota_N\vol$ on
  $S\setminus\car(S)$ simplifies to
  \begin{align*}
    \iota_N\vol
    &=\iota_N\left(\frac{n!}{2k^{2n+2}}\sum_{i=1}^{2n+1}(-1)^{i-1}x_i
      \bigwedge_{\substack{l=1 \\ l\not= i}}^{2n+2}\db x_l\right)\\
    &=\frac{n!}{2k^{2n+2}}
      \sum_{i=1}^{2n+1}(-1)^{i-1}x_i\left(\iota_N\db x_{2n+2}\right)
    \bigwedge_{\substack{l=1 \\ l\not= i}}^{2n+1}\db x_l\;.
  \end{align*}
  Due to
  \begin{displaymath}
    \iota_N\db x_{2n+2}
    =k\sqrt{\sum_{i=1}^{2n}x_i^2}
  \end{displaymath}
  the claimed result follows.
\end{proof}

To show that Lemma~\ref{lem:volinsphere} gives rise to the expression
for the volume form $\mu$ stated in Proposition~\ref{propn:vol},
it remains to change to spherical coordinates
$(r,\varphi_1,\dots,\varphi_{2n-1})$ for $S\setminus\car(S)$ where
$r\in(0,\frac{\pi}{k})$, $\varphi_1,\dots,\varphi_{2n-2}\in[0,\pi]$ and
$\varphi_{2n-1}\in[0,2\pi)$ are such that
\begin{align*}
  x_i&=\sin(kr)\cos(\varphi_i)\prod_{l=1}^{i-1}\sin(\varphi_l)
       \quad\text{for }i\in\{1,\dots,2n-1\}\;,\\
  x_{2n}&=\sin(kr)\prod_{l=1}^{2n-1}\sin(\varphi_l)\;,\\
  x_{2n+1}&=\cos(kr)\;.
\end{align*}
Using $\sin(kr)>0$ for $r\in(0,\frac{\pi}{k})$, we obtain
\begin{displaymath}
  \sqrt{\sum_{i=1}^{2n}x_i^2}=\sin(kr)\;.
\end{displaymath}
Comparing the expressions in Cartesian coordinates and in spherical
coordinates for the volume form of a $2n$-dimensional
Euclidean sphere, we deduce that
\begin{displaymath}
  \sum_{i=1}^{2n+1}(-1)^{i-1}x_i
  \bigwedge_{\substack{l=1 \\ l\not= i}}^{2n+1}\db x_l
\end{displaymath}
can be written as
\begin{displaymath}
  k\left(\sin(kr)\right)^{2n-1}\left(
  \prod_{i=1}^{2n-2}\left(\sin(\varphi_i)\right)^{2n-i-1}\right)
    \db r\wedge \db\varphi_1\wedge\dots\wedge\db\varphi_{2n-1}\;.
\end{displaymath}
The result claimed in Proposition~\ref{propn:vol} for
$(S^{2n+1},\distr,g)$ with parameter $k>0$ then follows from
Lemma~\ref{lem:volinsphere}.

In the last part, we discuss the random dynamics induced by
the operator $\frac{1}{2}\Delta$ on the characteristic foliation of
$S$. We have, with $\sdistr=\ker\quasi$,
\begin{displaymath}
  \left.\ker\db\quasi\right|_{\sdistr}=
  \operatorname{span}\left\{\frac{\pt}{\pt r}\right\}\;,
\end{displaymath}
which is implied by~(\ref{eq:quasiinS}) and
\begin{displaymath}
  \frac{\pt}{\pt r}=\sum_{i=1}^{2n+1}
  \frac{\pt x_i}{\pt r}\frac{\pt}{\pt x_i}
  =\frac{k}{\sqrt{\sum_{i=1}^{2n}x_i^2}}
  \left(\sum_{i=1}^{2n}x_ix_{2n+1}\frac{\pt}{\pt x_i}
    -\sum_{i=1}^{2n}x_i^2\frac{\pt}{\pt x_{2n+1}}\right)\;.
\end{displaymath}
We further compute
\begin{displaymath}
  g\left(\frac{\pt}{\pt r},\frac{\pt}{\pt r}\right)
  =\frac{1}{\sum_{i=1}^{2n}x_i^2}\left(\sum_{i=1}^{2n}x_i^2x_{2n+1}^2
    +\left(\sum_{i=1}^{2n}x_i^2\right)^2\right)
  =x_{2n+1}^2+\sum_{i=1}^{2n}x_i^2
  =1\;,
\end{displaymath}
showing that the representative $R=\frac{\pt}{\pt r}$ on
$S\setminus\car(S)$ of the
characteristic foliation of $S$ is a unit-length
vector field. From Proposition~\ref{propn:vol}, we obtain
\begin{displaymath}
  \operatorname{div}_\mu\left(R\right)
  =\frac{2n k\cos\left(kr\right)}{\sin\left(kr\right)}
  =2nk\cot\left(kr\right)\;.
\end{displaymath}
This establishes the second part of Theorem~\ref{thm:radial} that the
latitudinal process between the two
characteristic points on $S$ of the stochastic process with generator
$\frac{1}{2}\Delta$ on $S\setminus\car(S)$ follows the one-dimensional
diffusion process on $(0,\frac{\pi}{k})$ with generator
\begin{displaymath}
  \frac{1}{2}\frac{\pt^2}{\pt r^2}
  +nk\cot\left(kr\right)\frac{\pt}{\pt r}\;,
\end{displaymath}
that is, a Legendre process of order $2n+1$ and with parameter
$k>0$. Similarly to the observations
made in Section~\ref{sec:heis}, as a Legendre process of order $2n+1$
for $n\geq 1$ almost surely hits neither endpoint of the interval
$(0,\frac{\pi}{k})$, we deduce that the intrinsic sub-Laplacian
$\Delta$ on $S\setminus\car(S)$
is stochastically complete, whilst the
geometry induced on $S\setminus\car(S)$ is not geodesically complete.

\begin{remark}
  For $n=1$, the analysis presented above is in line with the
  discussions in~\cite[Section~5.1]{BBCH} for the sphere $S^2$
  embedded in $\mathrm{SU}(2)$, which is isomorphic to $S^3$, equipped
  with the standard sub-Riemannian contact structure.
\end{remark}

\subsection{\texorpdfstring{$\Sinads^{2n}$}{2n-hyperboloid}
  embedded in
  \texorpdfstring{$\ads^{2n+1}$}{(2n+1)-hyperboloid}}
\label{sec:hyper}
Our construction closely mimics the hyperboloid model for
hyperbolic space. Let
$(x_1,\dots,x_{2n+2})$ denote Cartesian coordinates on $\R^{2n+2}$
and consider the $(2n+1)$-dimensional hyperboloid $\ads^{2n+1}$
defined as
\begin{displaymath}
  \ads^{2n+1}=\left\{(x_1,\dots,x_{2n+2})\in\R^{2n+2}:
    \sum_{i=1}^{2n}x_i^2-x_{2n+1}^2-x_{2n+2}^2=-1
  \right\}\;.
\end{displaymath}
Fix $k\in\R$ positive and let $\eta$ be the Lorentzian metric on
$\R^{2n,2}$ given by
\begin{displaymath}
  \eta=\sum_{i=1}^{2n}\db x_i\otimes\db x_i
  -\db x_{2n+1}\otimes\db x_{2n+1}-\db x_{2n+2}\otimes\db x_{2n+2}\;.
\end{displaymath}
Using the contact form $\omega$ on the anti-de Sitter space
$\ads^{2n+1}$ defined by
\begin{equation}\label{eq:cf4hyper}
  \omega=\frac{1}{2k^2}\sum_{m=1}^{n+1}\left(
    x_{2m-1}\db x_{2m}-x_{2m}\db x_{2m-1}
  \right)\;,
\end{equation}
we get the contact structure $\distr=\ker\omega$ on $\ads^{2n+1}$
which we equip with the smooth fibre inner product $g$ obtained by setting,
for vector fields $X_1$ and $X_2$ in 
$\distr$,
\begin{equation}\label{eq:ginhyper}
  g(X_1,X_2)=\frac{1}{k^2}\,\eta(X_1,X_2)\;.
\end{equation}
This yields the standard
sub-Riemannian contact structure $(\distr,g)$ on
$\ads^{2n+1}$ subject to
an additional parameter $k>0$. Note that~(\ref{eq:ginhyper})
indeed defines a smooth fibre inner product $g$ on the contact
structure $\distr$ because the Reeb vector field $X_0$ on $H^{2n+1}$
given by
\begin{displaymath}
  X_0=2k^2\left(
    x_{2n+1}\frac{\pt}{\pt x_{2n+2}}-x_{2n+2}\frac{\pt}{\pt x_{2n+1}}-
  \sum_{m=1}^{n}\left(
    x_{2m-1}\frac{\pt}{\pt x_{2m}}-x_{2m}\frac{\pt}{\pt x_{2m-1}}
  \right)\right)
\end{displaymath}
is timelike due to
\begin{displaymath}
  \eta(X_0,X_0)=-4k^4\;,
\end{displaymath}
which implies that the distribution $\distr$ is spanned by spacelike
vector fields.

As in Section~\ref{sec:sphere}, the volume form $\vol$ on $\ads^{2n+1}$
given by~(\ref{eq:defnOmega}) can be expressed as
\begin{displaymath}
  \vol=\frac{n!}{2k^{2n+2}}
  \sum_{i=1}^{2n+2}(-1)^{i-1}x_i
  \bigwedge_{\substack{l=1 \\ l\not= i}}^{2n+2}\db x_l\;,
\end{displaymath}
and one can show as before that the
choice~(\ref{eq:cf4hyper}) for the contact form $\omega$ satisfies the
normalisation condition~(\ref{eq:normalise}).

Many of the subsequent computations are similar to the computations
performed in Section~\ref{sec:sphere}, but because the sub-Riemannian
metric in this section is obtained by restricting a Lorentzian
metric we choose to
treat these two families of model cases separately for clarity.

The hypersurface $S$ in $(\ads^{2n+1},\distr,g)$ which we study below is the
upper sheet of the hypersurface given by $\{x_{2n+2}=0\}$ or, phrased
differently,
\begin{displaymath}
  S=\left\{(x_1,\dots,x_{2n+2})\in \ads^{2n+1}:
    \sum_{i=1}^{2n}x_i^2-x_{2n+1}^2=-1\;,\enspace
    x_{2n+1}>0
  \right\}\;,
\end{displaymath}
that is, $S$ can be identified with the upper
sheet of a $2n$-dimensional two-sheeted hyperboloid. The hypersurface $S$ has a
unique characteristic point given by
\begin{displaymath}
  x_1=x_2=\dots=x_{2n}=0
  \quad\text{and}\quad
  x_{2n+1}=1\;,
\end{displaymath}
and inherits from the contact form $\omega$ on $\ads^{2n+1}$ the
quasi-contact form $\quasi$ on $S\setminus\car(S)$ 
defined by
\begin{equation}\label{eq:quasiinhyper}
  \quasi=\frac{1}{2k^2}\sum_{m=1}^{n}\left(
    x_{2m-1}\db x_{2m}-x_{2m}\db x_{2m-1}
  \right)\;.
\end{equation}

Analogous to Section~\ref{sec:sphere}, the
sub-Riemannian normal vector field $N$ along the hypersurface
$S\setminus\car(S)$ in $\ads^{2n+1}$ can be written as
\begin{displaymath}
  N=\frac{k}{\sqrt{\sum_{i=1}^{2n}x_i^2}}
  \left(\sum_{m=1}^n\left(
      x_{2m}x_{2n+1}\frac{\pt}{\pt x_{2m-1}}
      -x_{2m-1}x_{2n+1}\frac{\pt}{\pt x_{2m}}\right)
    +\sum_{i=1}^{2n}x_i^2\frac{\pt}{\pt x_{2n+2}}
  \right)
\end{displaymath}
and the volume form $\mu$ defined on $S\setminus\car(S)$ as
$\iota_N\vol$ takes the form
\begin{equation}\label{eq:volinhyper}
  \mu=\frac{n!}{2k^{2n+1}}\sqrt{\sum_{i=1}^{2n}x_i^2}
  \sum_{i=1}^{2n+1}(-1)^{i-1}x_i
  \bigwedge_{\substack{l=1 \\ l\not= i}}^{2n+1}\db x_l\;.
\end{equation}

In terms of the spherical coordinates
$(r,\varphi_1,\dots,\varphi_{2n-1})$ for $S\setminus\car(S)$ with
$r>0$, $\varphi_1,\dots,\varphi_{2n-2}\in[0,\pi]$ and
$\varphi_{2n-1}\in[0,2\pi)$, where
\begin{align*}
  x_i&=\sinh(kr)\cos(\varphi_i)\prod_{l=1}^{i-1}\sin(\varphi_l)
       \quad\text{for }i\in\{1,\dots,2n-1\}\;,\\
  x_{2n}&=\sinh(kr)\prod_{l=1}^{2n-1}\sin(\varphi_l)\;,\\
  x_{2n+1}&=\cosh(kr)\;,
\end{align*}
we have
\begin{displaymath}
  \sqrt{\sum_{i=1}^{2n}x_i^2}=\sinh(kr)\;.
\end{displaymath}
By further observing that
\begin{displaymath}
  \sum_{i=1}^{2n+1}(-1)^{i-1}x_i
  \bigwedge_{\substack{l=1 \\ l\not= i}}^{2n+1}\db x_l
\end{displaymath}
and
\begin{displaymath}
  k\left(\sinh(kr)\right)^{2n-1}
  \left(\prod_{i=1}^{2n-2}\left(\sin(\varphi_i)\right)^{2n-i-1}\right)
  \db r\wedge \db\varphi_1\wedge\dots\wedge\db\varphi_{2n-1}
\end{displaymath}
are expressions for the volume form on a $2n$-dimensional hyperboloid
in Cartesian coordinates and in spherical coordinates, respectively,
we deduce from~(\ref{eq:volinhyper}) that the volume form $\mu$ on
$S\setminus\car(S)$ embedded in
$(\ads^{2n+1},\distr,g)$ with parameter $k>0$ can be written as stated in
Proposition~\ref{propn:vol}. This concludes the proof of
Proposition~\ref{propn:vol}.

We further obtain
\begin{displaymath}
  \frac{\pt}{\pt r}=\sum_{i=1}^{2n+1}
  \frac{\pt x_i}{\pt r}\frac{\pt}{\pt x_i}
  =\frac{k}{\sqrt{\sum_{i=1}^{2n}x_i^2}}
  \left(\sum_{i=1}^{2n}x_ix_{2n+1}\frac{\pt}{\pt x_i}
    +\sum_{i=1}^{2n}x_i^2\frac{\pt}{\pt x_{2n+1}}\right)\;,
\end{displaymath}
which together with~(\ref{eq:quasiinhyper}) implies that, for
$\sdistr=\ker\quasi$,
\begin{displaymath}
  \left.\ker\db\quasi\right|_{\sdistr}=
  \operatorname{span}\left\{\frac{\pt}{\pt r}\right\}\;.
\end{displaymath}
Due to the fibre inner product $g$ on the contact structure $\distr$
arising by restricting a positive constant multiple of the Lorentzian
metric $\eta$ on $\R^{2n,2}$, we have
\begin{displaymath}
  g\left(\frac{\pt}{\pt r},\frac{\pt}{\pt r}\right)
  =\frac{1}{\sum_{i=1}^{2n}x_i^2}\left(\sum_{i=1}^{2n}x_i^2x_{2n+1}^2
    -\left(\sum_{i=1}^{2n}x_i^2\right)^2\right)
  =x_{2n+1}^2-\sum_{i=1}^{2n}x_i^2
  =1\;.
\end{displaymath}
It follows that the representative $R=\frac{\pt}{\pt r}$ on
$S\setminus\car(S)$ of the
characteristic foliation of $S$ is a vector field of
unit length. Using Proposition~\ref{propn:vol}, we compute
\begin{displaymath}
  \operatorname{div}_\mu\left(R\right)
  =\frac{2n k\cosh\left(kr\right)}{\sinh\left(kr\right)}
  =2nk\coth\left(kr\right)\;,
\end{displaymath}
which shows that the radial part of the stochastic process with
generator $\frac{1}{2}\Delta$ on $S\setminus\car(S)$ is the
one-dimensional diffusion process on $(0,\infty)$ with generator
\begin{displaymath}
  \frac{1}{2}\frac{\pt^2}{\pt r^2}
  +nk\coth\left(kr\right)\frac{\pt}{\pt r}\;.
\end{displaymath}
Since this yields a hyperbolic Bessel process of order $2n+1$ with
parameter $k>0$ that completes the proof of
Theorem~\ref{thm:radial}. Moreover, as in Section~\ref{sec:heis} and
Section~\ref{sec:sphere}, we conclude that the intrinsic sub-Laplacian
$\Delta$ on $S\setminus\car(S)$
is stochastically complete despite the
geometry induced on $S\setminus\car(S)$ not being geodesically
complete.
\begin{remark}
  Since the Lie group $\mathrm{SL}(2,\R)$ is isomorphic to the hyperboloid
  given by $x_1^2+x_2^2-x_3^2-x_4^2=-1$, see
  Wang~\cite[Remark~2.1]{jingwang}, we obtain the results
  from~\cite[Section~5.2]{BBCH} by taking $n=1$ in the above analysis.
  Chang, Markina and Vasil'ev~\cite{CMV}
  study further properties of the
  sub-Riemannian structure on the anti-de Sitter space $\ads^3$ as
  well as of a sub-Lorentzian structure on $\ads^3$.
\end{remark}

\bibliographystyle{alphaabbrv}
\bibliography{SR-hyper-biblio}

\end{document}